\documentclass[11pt]{amsart}
\usepackage{amsfonts}
\usepackage{graphicx}
\usepackage{amscd}

\usepackage{amsthm}
\usepackage{amstext}
\usepackage{amsmath}
\usepackage{amssymb}
\usepackage[mathscr]{eucal}
\usepackage{url}
\usepackage{verbatim}
\usepackage{mathtools}
\usepackage{booktabs}

\usepackage{tikz}

\usepackage[official]{eurosym}

\newtheorem{theorem}{Theorem}
\theoremstyle{plain}

\newtheorem{corollary}{Corollary}

\newtheorem{example}{Example}

\newtheorem{lemma}{Lemma}

\newtheorem{proposition}{Proposition}
\newtheorem{remark}{Remark}

\numberwithin{equation}{section}

\begin{document}

\title[Relationship between Lozi maps and max-difference equations]{On the relationship between Lozi maps and max-type difference equations} 
\author{A. Linero Bas and D. Nieves Rold\'an}
\begin{abstract}
In the present work we revise a transformation that links generalized Lozi maps with max-type difference equations. In this view, according to the technique of topological conjugation, we relate the dynamics of a concrete Lozi map with a complete uniparametric family of max-equations, and we apply this fact to investigate the dynamics of two particular families. Moreover, we present some numerical simulations related to the topic and, finally, we propose some open problems that look into the relationship established between generalized Lozi maps and max-equations.
\end{abstract}
\maketitle
 
\section{Introduction}

In 1963, the meteorologist E.N. Lorenz, \cite{Lorenz63}, studied a system of three first-order differential equations whose solutions stand out for being unstable with respect to small modifications, that is, minimal variations on the initial conditions of the system can produce a considerably different evolution of the terms. This property, which makes the trajectories wander in an apparently erratic manner, yields to what is known as a strange attractor.

The term \textit{strange attractor} was coined by Ruelle and Takens in \cite{RuelleT71} while studying the generation of turbulence\footnote{Based on \cite{RuelleT71}, we will say that a system generates turbulence if its motion becomes very complicated, irregular and chaotic.} in dissipative systems. In concrete, the analyzed system appeared to have an attractor which was locally the product of a Cantor set and a piece of two-dimensional manifold. Such structure was named by the authors as \textit{strange}. In an informal way, when an attractor is fractal, that is, its dimension is not an integer, we will called it strange attractor. In these cases, the dynamics on the attractor are said to be chaotic and they are sensitive to initial conditions (see \cite{Ott93}).

This phenomenon motivated the analysis of similar models that exhibit the same properties. For instance, in 1976, M. Hénon investigated the following two-dimensional mapping:
\begin{equation} \label{Eq_Henon}
    \left\{ \begin{matrix} x_{n+1} & = & y_n + 1 - ax_n^2 \\ y_{n+1} & = & bx_n \end{matrix}, \right.
\end{equation}

\noindent where $a, b$ are positive real numbers. In \cite{Henon76}, he simulated such model in the particular case $a=1.4$, $b=0.3$. Depending on the selection of the initial conditions, the solution of Equation (\ref{Eq_Henon}) either diverges to infinity or tends to a strange attractor, which appears to be the product of a one-dimensional manifold by a Cantor set. Later on, Benedicks and Carleson, \cite{BenedicksC91}, proved analytically the existence of a strange attractor; they considered Equation (\ref{Eq_Henon}) with $1<a<2$ and $b>0$, for a small $b$ and an $a$ close to $2$. In such case, there exists an hyperbolic fixed point\footnote{Roughly speaking, we will say that a fixed point is hyperbolic if the linearized equation at this point has no roots with absolute value equal to one. For more detail, consult \cite{Elaydi05}.} through which there are a stable and an unstable manifold. The closure of the unstable manifold is the attractor of the system.

In 1978, Lozi, \cite{Lozi}, exchanges the quadratic term in Hénon's map by considering the system of difference equation
\begin{equation} \label{Eq_Lozi_system}
    \left\{ \begin{matrix} x_{n+1} & = & 1 - a|x_n| + y_n \\ y_{n+1} &=& bx_n \end{matrix} \right.,
\end{equation}

\noindent where $a,b \in \mathbb{R}$. The numerical simulations developed for the particular case $a=1.7$ and $b=0.5$ suggested the existence of a strange attractor simpler than the one exhibited in Hénon map. In concrete, it seemed the product of parts of straight lines by a Cantor set. Lozi map gave raise to an abundant literature, as can be appreciated in the different chapters of this volume. Indeed, Lozi map was the first system for which the existence of a strange attractor was analytically established. In \cite{Misiurewicz80}, for certain values of the parameters $a$ and $b$, trapping regions\footnote{A trapping region is a nonempty set that is mapped with its closure into its interior.} were found and the author showed the hyperbolic structure of the map, which enabled him to prove that the intersection of the images of the trapping regions is a strange attractor. In fact, the restrictions that the parameters must verify in order to have such attractor are
\begin{equation}\label{E:conditions}
    0<b<1; \ \ a>b+1; \ \ 2a+b < 4; \ \ b < \frac{a^2-1}{2a+1}; \ \ \sqrt{2}a> b+ 2.
\end{equation}

Then, the attractor can be constructed from the successive forward iteration of a trapping region which is the triangle with vertices $I$, $F(I)$ and $F^2(I)$, where $F(x,y) = (1-a|x|+y, bx)$ and $I$ is the point given by the intersection of the unstable manifold of the fixed point $\left(\frac{1}{a+1-b}, \frac{b}{a+1-b}\right)$ with the horizontal axis, that is, $I= \left(\frac{2+a+\sqrt{a^2+4b}}{2(1+a-b)}, 0 \right)$.

It is worth mentioning that in \cite{BSV09} the authors even characterize the basin of attraction\footnote{The basin of attraction corresponds to the set of points of the plane whose orbits tend to the strange attractor.} for the strange attractor of the Lozi map whose parameters $a$ and $b$ verify the conditions established in (\ref{E:conditions}).

As an example of further research developed after Lozi map, we can highlight the particular case of Equation (\ref{Eq_Lozi_system}) with $a=-1$ and $b=-1$ that receives the name of \textit{Gingerbreadman equation}:
\begin{equation} \label{Eq_Devaney}
    \left\{ \begin{matrix} x_{n+1} &=& 1 + |x_n| + y_n \\ y_{n+1} &=& - x_n \end{matrix}, \right.
\end{equation}

\noindent studied by Devaney in \cite{Devaney84}. He proved that there exists a unique fixed point, $(1,1)$, which is elliptic\footnote{A fixed point is elliptic if his stability matrix has purely imaginary eigenvalues.}. Furthermore, there exists a hexagon where every point except the fixed point is periodic of period $6$. However, he showed that the fixed point is surrounded by infinitely many invariant polygons of arbitrarily large radius and the regions between any two of those consecutive polygons provide the equation of zones of instability. In this sense, Equation (\ref{Eq_Devaney}) is chaotic in certain regions and stable in others.

Moreover, in 1992, Crampin, \cite{Crampin92}, studied the piecewise linear equation 
\begin{equation} \label{Eq_Crampin}
    x_{n+1} = |x_n| - x_{n-1},
\end{equation}

\noindent which is globally periodic of period $9$.\footnote{A difference equation is called globally periodic of period $p$ when all the solutions are periodic of period (not necessarily minimal) $p$.} Also, he observed that each one of the linear difference equations, $x_{n+1} = x_n - x_{n-1}$ and $x_{n+1} = -x_n - x_{n-1}$, are periodic of periods $6$ and $3$, respectively. As a consequence he proposed the study of the combination of two periodic linear difference equations through a piecewise linear equation. In this direction, in \cite{Beardon95} Beardon et al.  considered the difference equation
\begin{equation} \label{Eq_Beardon}
    x_{n+1} = \alpha|x_n| + \beta x_n - x_{n-1},
\end{equation}

\noindent where $\alpha, \beta$ are real numbers, and studied for which values of the parameters $\alpha, \beta$, Equation (\ref{Eq_Beardon}) was globally periodic. In fact, they were able to prove that the set of points $(\alpha,\beta)$ for which the Equation (\ref{Eq_Beardon}) is globally periodic is unbounded and uncountable. For instance, they showed, \cite[Theorem 1.5.]{Beardon95}, that for $p\geq 2$, if $$\beta^2 - \alpha^2 = 2 \left(1+ \cos \frac{\pi}{p}\right), \ \ \beta < 0,$$ \noindent then Equation (\ref{Eq_Beardon}) is periodic with period $4p$. 

Furthermore, in \cite{AbuSaris99}, the authors gave necessary conditions to assure periodicity of
\begin{equation} \label{Eq_Abu_Saris}
    x_{n+1} = \alpha |x_n| + \beta x_{n-1},
\end{equation}

\noindent where $\alpha, \beta$ are real numbers. In concrete, for $\alpha \notin \{0,1\}$, it is necessary that $|\alpha| \in (0,2) \setminus \mathbb{Q}$ and $\beta =-1$, for Equation (\ref{Eq_Abu_Saris}) to be globally periodic.

In the literature, see for instance \cite{GroveL05}, Lozi map also appears in connection with a class of difference equations, so-called max-type equations. To have a general scope of the dynamics of different classes of max-type equations, consult the survey paper \cite{LineroN22}. In the present paper we are interested in deepening in the relationship between Lozi maps and some class of max-type difference equations. In this sense, in Section \ref{Sec_transformation}, we will show well-known changes of variables transforming a generalized Lozi map into max-type difference equations when some additional conditions to the parameters are considered. We will clarify the casuistic concerning the values of an arbitrary positive value $A$, fixed in advance. These transformations allow us to relate the dynamics of concrete generalized Lozi maps with that of some one-parametric families of max-type difference equations. In this direction, it is interesting to point out that all the members of this family share the same dynamics as they are topologically conjugate\footnote{A discrete dynamical system is a pair $(Y, \varphi)$, where $Y$ is a topological space and $\varphi:Y\rightarrow Y$ is continuous. In the case that $X$ is a topological space, we call associated dynamical system to the difference equation $x_{n+k} = f(x_{n+k-1}, \ldots, x_{n+1}, x_n)$ to $(X^k,F)$, where the map $F:X^k\rightarrow X^k$ is given by $F(x_1,\ldots,x_k)=(x_2,\ldots,x_k,f(x_k,\ldots,x_2,x_1))$. In this sense, we say that the dynamical systems $\varphi:X\rightarrow X$ and $\psi:Y\rightarrow Y$, where $X$ and $Y$ are topological spaces, are topologically conjugate if there is an homeomorphism $\phi: X \rightarrow Y$, such that $\phi(\varphi(x)) = \psi(\phi(x))$ for all $x \in X$. 
We say that two difference equations are topologically conjugate when the associated dynamical systems so are. Notice that, in this case, the difference equations exhibit the same type of dynamics; for instance, they have the same number of equilibrium points or periodic orbits, or have chaotic attractors which are homeomorphic,...} to the same generalized Lozi map. Afterwards, in Sections \ref{Sec_max} and \ref{S:caso_a=b}, we will apply this study in order to obtain some properties related to the dynamics of these families by interpreting them in the light of the conjugation of Lozi map and max-type equations. Then, Section \ref{Sec_simulation} collects some simulations and display some particular dynamics which are the leitmotiv to present some associated problems. Finally, we will present some conclusions in Section \ref{Sec_conclusions}.

\section{The transformation} \label{Sec_transformation}

As it can be read in the introduction of \cite{Elhadj14}, the idea of exchanging the quadratic term of Hénon map, see Equation (\ref{Eq_Henon}), into the absolute value function, came to Lozi's mind when he embedded the shape of the area generated by Hénon map, bounded by two parabolas, into another area bounded by four line segments, which reminded him to the graph of the absolute value function. Thus, in 1978, \cite{Lozi}, Lozi introduced the following system of difference equations
\begin{equation} \label{Eq_Lozi_system1}
    \left\{ \begin{matrix} x_{n+1} & = & 1 - a|x_n| + y_n \\ y_{n+1} &=& bx_n \end{matrix} \right.,
\end{equation}

\noindent where $a, b$ are positive real numbers. Such system, which can be reduced into the difference equation
\begin{equation} \label{Eq_Lozi}
    x_{n+1} = 1 - a|x_n| + bx_{n-1},
\end{equation}

\noindent was called \textit{Lozi map}. 

In the Lozi's original paper, the mapping is presented under the form of Equation~(\ref{Eq_Lozi_system1}). 
Additionally, in different papers the Lozi map is also given by
\begin{equation} \label{Eq_Lozi_system2}
    \left\{ \begin{matrix} x_{n+1} & = & 1 - a|x_n| + by_n \\ y_{n+1} &=& x_n \end{matrix}, \right.
\end{equation}
even, we can find the definition of a Lozi map as the bidimensional map 
\begin{equation} \label{Eq_Lozi_system3}
    \left\{ \begin{matrix} x_{n+1} & = & y_n \\ y_{n+1} &=&  1 - a|y_n| + x_n \end{matrix}. \right. 
\end{equation}
It is easy to establish that the three formulations are topologically conjugate, and for the sake of completeness we proceed to show the conjugations.  This observation will allow us to deal with one of them and to translate automatically its dynamics to the other ones. In our case, we will consider the system~(\ref{Eq_Lozi_system3}). 
 We denote by $ F_1,F_2,F_3$ the corresponding two-dimensional maps $F_j (x,y)$ associated to the systems (\ref{Eq_Lozi_system1}), (\ref{Eq_Lozi_system2}), (\ref{Eq_Lozi_system3}), respectively. To see that (\ref{Eq_Lozi_system1}) and (\ref{Eq_Lozi_system2}) are conjugate, consider the homeomorphism $\sigma(x,y)=(x,\frac{1}{b} y).$ Then, it is immediate to check that $F_1=\sigma^{-1}\circ F_2\circ \sigma.$ 
On the other hand, consider the homeomorphism $\tau(x,y)=(y,x).$ It is immediate to see that 
$F_3=\tau^{-1}\circ  F_2\circ \tau.$ 
Since the topological conjugation is an equivalence relation, we conclude that the three systems  (\ref{Eq_Lozi_system1}), (\ref{Eq_Lozi_system2}), (\ref{Eq_Lozi_system3}) are topologically conjugate.

A natural generalization of Equation (\ref{Eq_Lozi}) can be given by
\begin{equation} \label{Eq_gen_Lozi}
    y_{n+1} = \alpha |y_n| + \beta y_n + \gamma y_{n-1} + \delta,
\end{equation}

\noindent where $\alpha, \beta, \gamma, \delta$ are real numbers with $\alpha \neq 0$, otherwise we would obtain a linear difference equation whose dynamics is very well known. In the sequel we will refer to Equation (\ref{Eq_gen_Lozi}) as \textit{generalized Lozi map}. For instance, in the literature (see \cite{GroveL05}) we can find a version of such equation
\begin{equation} \label{Eq_Lozi_gen}
   y_{n+1} = \frac{k}{2} |y_n| + \left(\frac{k}{2}-l\right) y_n - my_{n-1} + \delta,
\end{equation}

\noindent where $k,l,m, \delta \in \mathbb{Z}$. 

Returning to Equation (\ref{Eq_gen_Lozi}), firstly, notice that in the particular case $\alpha =-a$, $\beta = 0$, $\gamma = b$ and $\delta = 1$, we recover Equation (\ref{Eq_Lozi}).

Now, we are going to show the general changes of variables which transform the generalized Lozi equation described in (\ref{Eq_gen_Lozi}) into a max-equation.

By means of a change of the form $$y_n = \log_A(z_n^q)+p$$ \noindent (where $p$ and $q$ are, in principle, arbitrarily taken real numbers with $q\neq 0$, and $A>0$ is also arbitrary, $A\neq 1$, with $z_n>0$ for all $n\geq -1$), and the observation that $|z|=\max\{z,-z\}$, Equation (\ref{Eq_gen_Lozi}) is transformed into (with $P=A^p$):
\small

\begin{equation} \label{Int}
    \log_A(z_{n+1}) = \frac{\alpha}{q} \max\{\log_A(z_n^q\cdot P), -\log_A(z_n^q\cdot P) \} + \log_A(z_n^\beta \cdot z_{n-1}^\gamma) +\frac{\beta p + \gamma p + \delta - p }{q}.
\end{equation}
\normalsize

Next, bearing in mind that the logarithm function $\log_A x$ is increasing (decreasing) when $A>1$ ($0<A<1$), we take an arbitrary $q$ holding $\frac{\alpha}{q} > 0$ ($\frac{\alpha}{q} < 0$). Then we can exchange the maximum function and the logarithm as $$\frac{\alpha}{q} \max\{ \log_A (z_n^q \cdot P), - \log_A(z_n^q \cdot P)\} = \log_A \max \left\{z_n^\alpha\cdot P^{\alpha/q}, \frac{1}{z_n^\alpha \cdot P^{\alpha/q}}\right\},$$ \noindent for $A>1$ and $\frac{\alpha}{q} > 0$; whereas 
\begin{eqnarray*}
&&\frac{\alpha}{q} \max\left\{ \log_A(z_n^q \cdot P), -\log_A (z_n^q \cdot P) \right\} \\ &=& \min \left\{ \log_A (z_n^\alpha \cdot P^{\alpha/q}, \log\left( \frac{1}{z_n^\alpha \cdot P^{\alpha/q}} \right) \right\} \\
&=& \log_A \left( \max\left \{ z_n^\alpha \cdot P^{\alpha/q}, \frac{1}{z_n^\alpha \cdot P^{\alpha/q}} \right\} \right),
\end{eqnarray*} \noindent for $0<A<1$ and  $\frac{\alpha}{q} < 0$.

We can gather the above discussion in the following result.

\begin{lemma} \label{Lemma_initial}
Let  $$y_{n+1} = \alpha |y_n| + \beta y_n + \gamma y_{n-1} + \delta$$ be a generalized Lozi map. Let $p$ and $q$ arbitrarily real numbers, and $A>0$. Then, the equation is transformed into 
\begin{equation} \label{EqI}
    \log_A (z_{n+1}) = \log_A \left( \frac{\max \left\{z_n^{2\alpha}, \frac{1}{A^{2\alpha p/q}} \right\} }{z_n^{\alpha - \beta}\cdot z_{n-1}^{-\gamma} } \right) + \frac{p(\alpha + \beta + \gamma - 1) + \delta }{q},
\end{equation}
if $A>1$ and $\frac{\alpha}{q} > 0$; or $0< A< 1$ and $\frac{\alpha}{q} < 0$.
\end{lemma}

\begin{remark} \label{Remark_conjugacion}
Notice that, by Lemma \ref{Lemma_initial}, Equations (\ref{Eq_gen_Lozi}) and (\ref{EqI}) are topologically conjugate. Indeed, it suffices to consider the homeomorphism $\phi: (0,\infty)^2 \rightarrow \mathbb{R}^2,$ $\phi(x,y)= (\alpha(x), \alpha(y))$, where $\alpha(u) = \log_A (u^q)+p$, thus $\alpha^{-1} (w)= A^{\frac{w-p}{q}}$. Then, if we denote by $F$ the discrete dynamical system in $\mathbb{R}^2$ given by $F(x,y)=(y, \alpha|y| + \beta y + \gamma x + \delta)$, and by $\tilde{F}$ the system defined in $(0,\infty)^2$ as $\tilde{F}(x,y)= \left(y, \frac{\max\left\{y^{2\alpha}, A^{-\frac{2\alpha p}{q}} \right\}}{y^{\alpha - \beta}\cdot x^{-\gamma}}\cdot A^{\frac{p(\alpha + \beta + \gamma - 1) + \delta }{q}} \right) $, it is easily seen that $\tilde{F}=\phi^{-1} \circ F \circ \phi$.
\end{remark}

Some consequences of Lemma \ref{Lemma_initial} are the following:

\begin{proposition}
If $\delta = 0$, then the generalized Lozi map $$y_{n+1} = \alpha |y_n| + \beta y_n + \gamma y_{n-1} $$ is topologically conjugate to the max-type equation \begin{equation} \label{prop2}
z_{n+1} = \frac{\max\{z_n^{2\alpha}, B\} }{z_n^{\alpha - \beta}\cdot z_{n-1}^{-\gamma} } \cdot B^{\frac{\alpha + \beta + \gamma - 1}{-2\alpha}} 
\end{equation} for all $B>0$.

In particular, if $\alpha + \beta + \gamma - 1 = 0$, then the generalized Lozi map is topologically conjugate to the max-type equation 
\begin{equation} \label{prop1}
    z_{n+1} = \frac{\max \{ z_n^{2\alpha}, B\} }{z_n^{\alpha - \beta} \cdot z_{n-1}^{-\gamma}}
\end{equation} 
for all $B>0$.
\end{proposition}

\begin{proof}
If $\delta = 0$ and $\alpha + \beta + \gamma - 1 \neq 0$, then (\ref{EqI}) reads as follows $$\log_A (z_{n+1}) = \log_A\left( \frac{\max\left\{z_n^{2\alpha}, \frac{1}{A^{2\alpha p/q} } \right\} }{z_n^{\alpha-\beta} \cdot z_{n-1}^{-\gamma} } \right) + \frac{p(\alpha + \beta + \gamma - 1) }{q}.$$

If we take $p=0$ in the above equation, we find $$z_{n+1} = \frac{\max\{ z_n^{2\alpha}, 1\} }{z_n^{\alpha-\beta} \cdot z_{n-1}^{-\gamma} }. $$ 

If, otherwise, $p\neq 0$, the equation presents this writing $$z_{n+1} = \frac{\max\left\{z_n^{2\alpha}, \frac{1}{A^{\frac{2\alpha p}{q}}} \right\} }{z_n^{\alpha - \beta} \cdot z_{n-1}^{-\gamma} } \cdot A^{\frac{p(\alpha + \beta + \gamma - 1)}{q} } = \frac{\max\left\{z_n^{2\alpha}, A^{\frac{-2\alpha p}{q}} \right\} }{z_n^{\alpha - \beta} \cdot z_{n-1}^{-\gamma} } \cdot A^{ \left(\frac{-2\alpha p}{q}\right) \cdot \frac{\alpha + \beta + \gamma - 1}{-2\alpha} }.$$

When $A>1$ and $\frac{\alpha}{q} > 0$, since $p$ is arbitrarily taken, then once fixed the value $q$, having the same sign as $\alpha$, we have that $A^{-\frac{2\alpha p}{q}} \in (0, \infty)$. Therefore, if we write $B= A^{-\frac{2\alpha p}{q}}$, we obtain (\ref{prop2}).
 We can proceed analogously for $0<A<1$ and $\frac{\alpha}{q} < 0$.

Moreover, if $\alpha + \beta + \gamma - 1 = 0$, we can choose the value $p$ \textit{ad libitum}, and we have 
$$z_{n+1} = \max \left\{ z_n^{2\alpha}, \frac{1}{A^{2\alpha p/q} } \right\} \cdot z_n^{\beta - \alpha} \cdot z_{n-1}^\gamma  = \frac{\max \left\{ z_n^{2\alpha}, \left(\frac{1}{A}\right)^{\frac{2\alpha p}{q} } \right\}}{z_n^{\alpha - \beta}\cdot z_{n-1}^{-\gamma}}.$$

Since the choice of $p$ is arbitrary, if $A>1$ and $\frac{\alpha}{q} > 0$, we know that the range of $\left(\frac{1}{A}\right)^{\frac{2\alpha p}{q}}$ is $(0,\infty)$, and we obtain $$z_{n+1} = \frac{\max \{ z_n^{2\alpha}, B\} }{z_n^{\alpha - \beta} \cdot z_{n-1}^{-\gamma}}$$ for all $B>0$. The same applies if $0<A<1$ and $\frac{\alpha}{q} < 0$.
\end{proof}

\begin{example}In particular, $y_{n+1} = \frac{3}{2}|y_n| + \frac{1}{2}y_n - y_{n-1}$ is topologically conjugate to any max-type equation $$z_{n+1} = \frac{\max\{ z_n^3, B\} }{z_n \cdot z_{n-1} }, \ \ \ B>0.$$

On the other hand, $y_{n+1} = |y_n| + y_n + y_{n-1}$ is topologically conjugate to any max-type equation $$z_{n+1} = \frac{1}{B} \max\{z_n^2, B\} \cdot z_{n-1}, \ \ \ B>0.$$ $\hfill\Box$
\end{example}

\begin{proposition} \label{Prop_deltaneq0}
Let us consider 
\begin{equation} \label{Eq_Lozi_prop3}
    y_{n+1} = \alpha |y_n| + \beta y_n + \gamma y_{n-1} + \delta,
\end{equation}
with $\delta \neq 0$. Then, either for $A>1$ and $\frac{\alpha}{q}>0$, or for $0<A<1$ and $\frac{\alpha}{q}<0$, with $q\in \mathbb{R} \setminus \{0\}$, Equation (\ref{Eq_Lozi_prop3}) is topologically conjugate to $$z_{n+1} = \frac{\max\{z_n^{2\alpha}, A^{\frac{-2\alpha p}{q} } \} }{z_n^{\alpha - \beta} \cdot z_{n-1}^{-\gamma} } \cdot A^{\frac{p(\alpha+\beta + \gamma -1) + \delta }{q}},$$ for all $p\in \mathbb{R}$. In particular, assuming that $\alpha+\beta+\gamma-1\neq 0$, Equation (\ref{Eq_Lozi_prop3}) is topologically conjugate to:
\begin{itemize}
    \item $z_{n+1} = \frac{\max\{z_n^{2\alpha}, B \} }{z_n^{\alpha - \beta} \cdot z_{n-1}^{-\gamma} }$, for all $B>1$, if $\frac{\delta}{\alpha + \beta + \gamma -1}>0$. 
    \item $z_{n+1} = \frac{\max\{z_n^{2\alpha}, C \} }{z_n^{\alpha - \beta} \cdot z_{n-1}^{-\gamma} }$, for all $0<C<1$, if $\frac{\delta}{\alpha + \beta + \gamma -1}<0$.
\end{itemize}
Moreover, if $\alpha + \beta + \gamma -1 = 0$, additionally we get that Equation (\ref{Eq_Lozi_prop3}) is topologically conjugate to:
\begin{itemize}
    \item $z_{n+1} = \frac{\max\{z_n^{2\alpha}, 1 \} }{z_n^{\alpha - \beta} \cdot z_{n-1}^{-\gamma}}\cdot B$, for all $B>1$, if $\frac{\delta}{\alpha}>0$. 
    \item $z_{n+1} = \frac{\max\{z_n^{2\alpha}, 1 \} }{z_n^{\alpha - \beta} \cdot z_{n-1}^{-\gamma}} \cdot C$, for all $0<C<1$, if $\frac{\delta}{\alpha}<0$.
\end{itemize}
\end{proposition}

\begin{proof}
Firstly, by Lemma \ref{Lemma_initial}, it follows directly that Equation (\ref{Eq_Lozi_prop3}) is topologically conjugate to 
\begin{equation} \label{Eq_proof_prop3}
z_{n+1} = \frac{\max\{z_n^{2\alpha}, A^{\frac{-2\alpha p}{q} } \} }{z_n^{\alpha - \beta} \cdot z_{n-1}^{-\gamma} } \cdot A^{\frac{p(\alpha+\beta + \gamma -1) + \delta }{q}},  
\end{equation}
for $A>1$ and $\frac{\alpha}{q}>0$, or for $0<A<1$ and $\frac{\alpha}{q}<0$, and for all $p\in \mathbb{R}$.

Now, if $\alpha+\beta+\gamma-1\neq 0$, by taking $p = \frac{-\delta}{\alpha + \beta + \gamma - 1}$, (\ref{Eq_proof_prop3}) reduces to $$z_{n+1} = \frac{\max\left\{z_n^{2\alpha}, A^{\frac{2\alpha \delta}{q(\alpha + \beta + \gamma - 1)}} \right\} }{z_n^{\alpha - \beta} \cdot z_{n-1}^{-\gamma} }.$$

Now, depending on the sign of $\frac{\delta}{\alpha + \beta + \gamma - 1}$ and taking into account if $A>1$ or $0<A<1$, we deduce that the corresponding generalized Lozi map is topologically conjugate to:
\begin{itemize}
    \item The equations $z_{n+1} = \frac{\max\left\{z_n^{2\alpha}, B \right\} }{z_n^{\alpha - \beta} \cdot z_{n-1}^{-\gamma}}$ for all $B>1$ when $\frac{\delta}{\alpha +\beta + \gamma - 1} > 0$.
    \item The equations $z_{n+1} = \frac{\max\left\{z_n^{2\alpha}, C \right\} }{z_n^{\alpha - \beta} \cdot z_{n-1}^{-\gamma}}$ for all $0<C<1$ when $\frac{\delta}{\alpha +\beta + \gamma - 1} < 0$.
\end{itemize}

Finally, if $\alpha + \beta + \gamma -1 = 0$, then, Equation (\ref{Eq_proof_prop3}) is converted into $$z_{n+1} = \frac{\max\left\{ z_n^{2\alpha}, A^{-\frac{2\alpha p}{q} } \right\} }{z_n^{\alpha - \beta}\cdot z_{n-1}^{-\gamma}  } A^{\delta/q}. $$

So, if we take $p=0$, we arrive to the difference equation $$z_{n+1} = \frac{\max\{z_n^{2\alpha}, 1 \} }{z_n^{\alpha - \beta} \cdot z_{n-1}^{-\gamma} } \cdot A^{\delta/q}.$$

Since $A$ is arbitrary, depending on the interval where $A$ belongs and the corresponding sign of $q$, we find:
\begin{itemize}
    \item The equations $z_{n+1} = \frac{\max\{z_n^{2\alpha}, 1 \} }{z_n^{\alpha - \beta} \cdot z_{n-1}^{-\gamma} } \cdot B$ for all $B>1$ when $\frac{\delta}{\alpha} > 0$.
    \item The equations $z_{n+1} = \frac{\max\{z_n^{2\alpha}, 1 \} }{z_n^{\alpha - \beta} \cdot z_{n-1}^{-\gamma} } \cdot C$ for all $0<C<1$ when $\frac{\delta}{\alpha} < 0$.
\end{itemize}

\end{proof}

\begin{example}
For instance, $y_{n+1} = \frac{1}{2}|y_n| - \frac{1}{2}y_n - 2y_{n-1} + 1$ is topologically conjugate to $$z_{n+1} = \frac{\max\{z_n,1\} }{z_n \cdot z_{n-1}^2}\cdot B,$$ \noindent for all $B>1$. $\hfill\Box$
\end{example}

Finally, in order to show some conditions that allow us to connect Equation (\ref{Eq_gen_Lozi}) with max-type difference equations, we recover Equation (\ref{Eq_Lozi_gen}) introduced at the beginning of the section, that is, $$ y_{n+1} = \frac{k}{2} |y_n| + \left(\frac{k}{2}-l\right) y_n - my_{n-1} + \delta.$$

Obviously, the two formulations of Equation (\ref{Eq_gen_Lozi}) and Equation (\ref{Eq_Lozi_gen}) are equivalent, since it suffices to take into account that the equalities $\frac{k}{2} = \alpha$, $\frac{k}{2}-l = \beta$, $-m = \gamma$ and $\delta = \delta$ define a bijective relation between the coefficients of (\ref{Eq_gen_Lozi}) and (\ref{Eq_Lozi_gen}). In this sense, we find:
\begin{itemize}
    \item $\delta = k - 1 - l - m \Leftrightarrow \delta = \alpha + \beta + \gamma - 1$;
    \item $\delta = -(k - 1 - l - m) \Leftrightarrow \delta = -(\alpha + \beta + \gamma) + 1$;
    \item $\delta = 0$.
\end{itemize}

It should be emphasized that the above conditions $\delta = k - 1 - l - m$, $\delta = -(k - 1 - l - m)$ and $\delta = 0$ appear in \cite{GroveL05} when applying the change of variables to Equation (\ref{Eq_Lozi_gen}) in order to obtain a max-type difference equation, which is
\begin{equation} \label{MaxLozi}
    x_{n+1} = \frac{\max\{x_n^k, A \}}{x_n^l \cdot x_{n-1}^m}.
\end{equation}
In the light of our study, we can establish:
\begin{corollary}(\cite{GroveL05})
Let $k,l,m,\delta\in\mathbb R$. Then, the generalized Lozi map $ y_{n+1} = \frac{k}{2} |y_n| + \left(\frac{k}{2}-l\right) y_n - my_{n-1} + \delta$ is topologically conjugate to:
\begin{itemize}
\item $x_{n+1}=\frac{\max\{x_n^k, A\}}{x_n^l\cdot x_{n-1}^m}$ for all $A>0$, if $\delta=0$; 
\item $x_{n+1}=\frac{\max\{x_n^k, A\}}{x_n^l\cdot x_{n-1}^m}$ for all $A>1$, if $\delta=k - 1 - l - m$;
\item $x_{n+1}=\frac{\max\{x_n^k, A\}}{x_n^l\cdot x_{n-1}^m}$ for all $0<A<1$, if $\delta=-(k - 1 - l - m)$.
\end{itemize}
\end{corollary}

The transformation that links Equation (\ref{Eq_Lozi_gen}) with Equation (\ref{MaxLozi}) can be found in \cite{GroveL05} or \cite{Feuer03}. However, we have made a more general change of variables and the parameters are not restricted to the integer set $\mathbb Z$  as they are in the references cited, but they can be arbitrary real numbers. Furthermore, as far as we are concerned, it is the first time that the problem is treated from the point of view of topological conjugacies.

On the other hand, recall that we recover Lozi Equation (\ref{Eq_Lozi}) from the generalized Lozi equation (\ref{Eq_gen_Lozi}) if $\alpha = -a$, $\beta = 0$, $\gamma = b$ and $\delta = 1$. In this sense, from the above conditions, it yields that $\delta = \alpha + \beta + \gamma - 1$ implies $b-a=2$; and $\delta = -(\alpha + \beta + \gamma) + 1$ implies $a=b$. So, this motivates to analyze in detail Lozi map in the particular cases $a=b$ and $b-a=2$. Nevertheless, we will restrict our attention to case $a=b$.

\section{A family of max-type difference equations} \label{Sec_max}

In the present section we will analyze a concrete family of max-type difference equations. Our purpose with such study is to illustrate that, thanks to the transformations developed in the previous section, the analysis of a concrete Lozi map is sufficient to know the behaviour of a whole family of equations.

In this direction, we will deal with the particular one-parametric family of max-type difference equations $$x_{n+1} = \frac{\max\{ x_n^3, A\} }{x_nx_{n-1} }, \ \ \ A>0, $$ \noindent which is topologically conjugate to the generalized Lozi map
\begin{equation} \label{Eq_Lex} 
y_{n+1} = \frac{3}{2}|y_n| + \frac{1}{2} y_n - y_{n-1}.
\end{equation}

Notice that Equation (\ref{Eq_Lex}) is a particular case of Equation (\ref{Eq_gen_Lozi}) with $\delta = 0$ and $\alpha + \beta + \gamma = 1$. The dynamics of (\ref{Eq_Lex}) is easily described by the following results. The proof of the first one is straightforward and will be omitted.

\begin{proposition}
Equation (\ref{Eq_Lex}) has infinitely many equilibrium points, namely, $\mathcal{F} = \{ \bar{x}: \bar{x}\geq 0\}$,
\end{proposition}

\begin{theorem}
Every solution of Equation (\ref{Eq_Lex}) which is not an equilibrium point diverges to $\infty$.
\end{theorem}

\begin{proof}
We distinguish several cases depending on the values of the initial conditions.

\textbf{Case $1$:} If the initial conditions $(y_{-1},y_0)$ of (\ref{Eq_Lex}) hold $0\leq y_{-1} < y_0$, then by induction it is easily seen that $y_n = y_{-1} + (n+1)(y_0-y_{-1})$ for all $n\geq 1$, and, hence, $\lim_{n\rightarrow +\infty} y_n = +\infty$.

\textbf{Case $2$:} Similarly, if the initial conditions $(y_{-1}, y_0)$ satisfy $y_{-1} \leq 0 < y_0$ or $y_{-1} < 0 \leq y_0$, then $y_1 = 2y_0 - y_{-1} > y_0 > 0$, and we can apply Case $1$ to affirm that the solution goes to $+\infty$ when $n$ tends to $+\infty$.

\textbf{Case $3$:} Again, the solution is unbounded if the initial conditions verify $y_{-1} \leq y_0 \leq 0$, $|y_{-1}| + |y_0| >0$, since now we find $y_1 = -\frac{3}{2}y_0 + \frac{1}{2}y_0 - y_{-1} = -y_0-y_{-1}>0$ and we can apply Case $2$.

\textbf{Case $4$:} Suppose that $0\leq y_0 < y_{-1}$. Consider the value $\varepsilon = y_{-1} - y_0$ and fix the smallest non negative integer $N$ such that $$y_0 - (N+1)\varepsilon < 0 \leq y_0 - N\varepsilon.$$ Then, it is a simple matter to check that $$y_1 = 2y_0 - y_{-1} = y_0 - \varepsilon, \ldots, y_{N} = y_0 - N\varepsilon, y_{N+1} = y_0 + (N+1)\varepsilon,$$ \noindent and being $y_{N+1} < 0 \leq y_{N}$, we obtain 
\begin{eqnarray*}
y_{N+2} &=& \frac{3}{2}(-y_0 + (N+1)\varepsilon) + \frac{1}{2}(y_0 - (N+1)\varepsilon) - y_0 + N\varepsilon \\
&=& -2y_0 + (2N+1)\varepsilon.
\end{eqnarray*}

Now, if $y_{N+2} \geq 0$, we finish because we have the pair of new initial conditions $(y_{N+1}, y_{N+2})$, with $y_{N+1}<0\leq y_{N+2}$ and we apply Case $2$. But if $y_{N+2} < 0$, following the iteration we deduce that 
\begin{eqnarray*}
y_{N+3} &=& \frac{3}{2}(2y_0 - (2N+1)\varepsilon)+\frac{1}{2}(-2y_0 + (2N+1)\varepsilon)-(y_0 - (N+1)\varepsilon) \\
&=& y_0 - N\varepsilon \geq 0,
\end{eqnarray*}

\noindent and we can apply Case $2$ to the new initial conditions $y_{N+2} < 0 \leq y_{N+3}$.

\textbf{Case $5$:} For the case $y_0 \leq y_{-1} \leq 0$, $|y_{-1}| + |y_0| >0$, take into account that $y_1 = -y_0 - y_{-1} > 0$ and use Case $2$ to the initial conditions $y_0 \leq 0 < y_1$.

\textbf{Case $6$:} With respect to the last case of our discussion, we have $y_1 = -y_0 - y_{-1}$.

If $y_{-1}+y_0 < 0$, then $y_0 \leq 0 < y_1$ and we finish by a simple application of Case $2$. If, on the contrary, $y_{-1} + y_0 \geq 0$, then $y_0 < 0$, $y_1\leq 0$, and it suffices to apply Case 3 or Case 5 depending on whether $y_0\leq y_1\leq 0$ or $y_1\leq y_0\leq 0$, respectively. 

\end{proof}

\begin{corollary}
Given an arbitrary value $A \in (0, + \infty)$, consider the max-type difference equation $$x_{n+1} = \frac{\max \{x_n^3, A\}}{x_nx_{n-1}}.$$ For any arbitrary positive initial conditions $(x_{-1},x_0)$, either $x_{-1}=x_0 = \bar{x}$ is an equilibrium point, or the solution generated by them diverges to infinity. The stationary solution $(\bar{x}, \bar{x}, \bar{x}, \ldots)$ appears for all the values $\bar{x} \geq \sqrt[3]{A}$.
\end{corollary}

In view of this result, a question that may be of some interest is to study the dynamics of generalized Lozi map in the case $\delta = 0$ and $\alpha + \beta + \gamma = 1$.

\section{Lozi map for $a=b$}\label{S:caso_a=b}

Now, we will focus on the Lozi map, in the particular case where the parameters verify $a=b$,
\begin{equation}\label{Eq_Lozi_a=b}
    x_{n+1} = 1 - a|x_n| + ax_{n-1}.
\end{equation}

For such equation it is straightforward to determine its fixed points and we will omit its proof.

\begin{lemma}\label{L:points}
The equilibrium points of Equation (\ref{Eq_Lozi_a=b}) are given by:
\begin{itemize}
    \item[a)] $\bar{x}=1$ if $a\in \left(-\infty, \frac{1}{2}\right]$.
    \item[b)] $\bar{x}=1$ and $\bar{x} = \frac{1}{1-2a}$ if $a\in \left(\frac{1}{2},+\infty\right)$.
\end{itemize}
\end{lemma}

It is well-known (see \cite[pages 181-182]{Elhadj14}) that if $|a| + |b| <1$, then the Lozi map (\ref{Eq_Lozi}) has a unique fixed point which is globally attractor. In the case of $a=b$, such condition reads as $|a| < \frac{1}{2}$. Therefore,

\begin{proposition}
Let $a \in \left(-\frac{1}{2}, \frac{1}{2}\right)$. Then, all the solutions of Equation (\ref{Eq_Lozi_a=b}) converge to the equilibrium point $\bar{x}=1$.
\end{proposition}

In view of the fact that $a=\frac{1}{2}$ and $a=-\frac{1}{2}$ are boundary values for the condition $|a|<\frac{1}{2}$, the study of the difference equations $$x_{n+1} = 1 - \frac{1}{2}|x_n| + \frac{1}{2}x_{n-1} \ \ \text{and} \ \ x_{n+1} = 1 + \frac{1}{2}|x_n| - \frac{1}{2}x_{n-1}$$ \noindent can merit to pay our attention.

In the case of $a=\frac{1}{2}$, it seems that all the solutions, other than periodic solutions, converge to a $2$-periodic solution. We will prove it later.

Concerning periodic solutions of Lozi Equation (\ref{Eq_Lozi_a=b}), we have to discard the interval of values $a \in \left(-\frac{1}{2}, \frac{1}{2}\right)$ because in this segment of values we know that all the solutions converge to the global attractor $\bar{x}=1$. For $|a|>\frac{1}{2}$, by a straightforward way, we find:
\begin{lemma}
Assume that $|a|> \frac{1}{2}$. Then, the only initial conditions $(x_{-1},x_0)$ which generate $2$-periodic solutions are given by:
$$(x_{-1},x_0) = \left(\frac{1}{2a^2-2a+1}, \frac{1-2a}{2a^2-2a+1} \right), $$ or $$ (x_{-1},x_0) = \left(\frac{1-2a}{2a^2-2a+1} , \frac{1}{2a^2-2a+1}\right)$$
with $a>\frac{1}{2}$.
\end{lemma}

In this case, notice that the initial conditions have different signs. 

Concerning the stability of these $2$-periodic points for the Lozi map $F(x,y) = (y, 1-a|y|+ax)$, we have $$DF\left(\frac{1}{2a^2-2a+1}, \frac{1-2a}{2a^2-2a+1} \right) = \begin{pmatrix} 0 & 1 \\ a & a \end{pmatrix},$$ \noindent and $$DF\left( \frac{1-2a}{2a^2-2a+1}, \frac{1}{2a^2-2a+1} \right) = \begin{pmatrix} 0 & 1 \\ a & -a \end{pmatrix}.$$ Therefore, $$DF^2\left(\frac{1}{2a^2-2a+1}, \frac{1-2a}{2a^2-2a+1} \right) = \begin{pmatrix} 0 & 1 \\ a & a \end{pmatrix}\cdot \begin{pmatrix} 0 & 1 \\ a & -a \end{pmatrix} = \begin{pmatrix} a & -a \\ a^2 & a-a^2 \end{pmatrix}. $$

The corresponding eigenvalues are computed by the characteristic equation $$(-a^2+a-\lambda)\cdot(a-\lambda)+a^3 = 0,$$ $$\lambda^2 + (a^2-2a)\lambda + a^2 = 0,$$ \noindent whose roots are given by $$\lambda_1 = \frac{2a-a^2 + a\sqrt{a^2-4a}}{2}, \ \lambda_2 = \frac{2a-a^2 - a\sqrt{a^2-4a}}{2},$$

\noindent and these roots lie inside the unit circle if and only if the coefficients of the characteristic polynomial verify (consult \cite[Th.2.37]{Elaydi05}): $$|-a^2+2a| < 1+a^2 < 2,$$ \noindent which is obviously true for the range of values $a\in \left(\frac{1}{2}, 1\right)$. This means that these $2$-periodic points are locally asymptotically stable\footnote{Roughly speaking, we say that $2$-periodic points are locally asymptotically stable if they are locally stable and for every set of initial conditions in a certain neighborhood of the periodic points they converge to the periodic orbit. For a more accurate definition, consult for instance \cite{Elaydi05}. } when $\frac{1}{2} < a < 1$. 

\subsection{Case $a=\frac{1}{2}$}

We consider the parameter $a=\frac{1}{2}$ for the Lozi map
\begin{equation} \label{Eq_Lozi_a0.5}
    x_{n+1} = 1 - \frac{1}{2}|x_n| + \frac{1}{2}x_{n-1}.
\end{equation}

The max-difference equation associated to Equation (\ref{Eq_Lozi_a0.5}) is given by (use Proposition~\ref{Prop_deltaneq0})
\begin{equation} \label{EMax}
    x_{n+1} = x_n^{1/2}\cdot x_{n-1}^{1/2} \cdot \max\left\{ \frac{1}{x_n}, A \right\}, \ \ \ 0 < A < 1.
\end{equation}

The following results establish the fixed points and the $2$-periodic orbits of Equation (\ref{Eq_Lozi_a0.5}), respectively. The proof of the first one is immediate and will be omitted. 

\begin{lemma}
The unique equilibrium point of (\ref{Eq_Lozi_a0.5}) is $\bar{x} = 1$.
\end{lemma}

\begin{lemma}
Suppose that the initial condition $(x_{-1}, x_0) = (x,y)$ generates a $2$-periodic solution of Equation (\ref{Eq_Lozi_a0.5}). Then, $$0\leq x,y \leq 2; \ \ x+y=2; \ \ x\neq y.$$
\end{lemma}

\begin{proof}
Assume that the pair $(x,y)$ provides a $2$-periodic sequence. From the corresponding iteration, we deduce that 
\begin{equation} \label{1}
    \left\{ \begin{matrix} x &=& 1 - \frac{1}{2} |y| + \frac{1}{2}x \\ y &=& 1 - \frac{1}{2}|x| + \frac{1}{2}y \end{matrix}. \right. 
\end{equation}

Equivalently, $\frac{1}{2}x = 1 - \frac{1}{2}|y|$ and $\frac{1}{2}y = 1 - \frac{1}{2}|x|$. Now, if we subtract both equations, we deduce
\begin{equation} \label{3}
    x-y = |x| - |y|.
\end{equation}

Next, we do a proof by cases.

$\bullet$ If $x,y \geq 0$, suppose that $x = \alpha \cdot y$, for some $\alpha \geq 0$. If $\alpha = 0$, $x=0$, we deduce that $y=2$, and it is immediate to check that the initial conditions $(x_{-1},x_0)=(0,2)$ generates a $2$-periodic solution. If $\alpha >0$, we find $\frac{1}{2}x = 1- \frac{1}{2} \frac{x}{\alpha}$; or $x=\frac{2\alpha}{1+\alpha}$. Then, $y = \frac{x}{\alpha} = \frac{2}{1+\alpha}$.

Notice that $x+y=2$. It is easy to see that $\left( \frac{2\alpha}{1+\alpha}, \frac{2}{1+\alpha} \right)$ generates a $2$-periodic orbit.

$\bullet$ If $x\geq 0, y \leq 0$, from (\ref{3}) we deduce that $y=0$, and therefore $x=2$. The initial conditions are $(2,0)$, already included in the above case. By symmetry, the case $x\leq 0, y\geq 0$ leads us to the same conclusion.

$\bullet$ If $x\leq 0, y \leq 0$, from (\ref{3}) we have $x-y=0$, or $x=y$. Replacing this equality in (\ref{1}) yields to $\frac{1}{2}x = 1 + \frac{1}{2}x$, which is impossible.
\end{proof}

Now, we return to the max-equation, for instance through the change of variables $y_n = \log_A \left(\frac{A}{x_n^k}\right) = \log_A(Ax_n) = 1 - \log_A x_n$ (we take $p=1, q=-1$),  where $(x_n)$ and $(y_n)$ are the corresponding solutions of Equation (\ref{EMax}) and the Lozi equation, respectively. It is direct to see then that the following result holds.

\begin{corollary}
Suppose that the initial conditions $(x_{-1}, x_0) = (x,y)$ generate a $2$-periodic solution of Equation (\ref{EMax}). Then, $$x,y \in \left[A,\frac{1}{A}\right], \ \ A \in (0,1) \ \ \text{with} \ \ xy=1, \ x\neq y.$$
\end{corollary}

Realize that we have included the condition $x\neq y$ in order to avoid the appearance of the equilibrium point $\bar{x}=1$.

Next, we are going to prove that any initial condition $(x_{-1}, x_0)$ generates by Equation~(\ref{Eq_Lozi_a0.5}) an asymptotically periodic orbit to some appropriate $2$-periodic orbit $(x,y,x,y,\ldots)$ with $x+y=2$.

Concerning the analysis of the stability for $2$-periodic orbits of $x_{n+1} = 1 - \frac{1}{2}|x_n| + \frac{1}{2}x_{n-1}$, since $F(x,y) = (y, 1 - \frac{1}{2}|y|+\frac{1}{2}x)$, and the coordinates are positive, it is $DF(x,y) = \begin{pmatrix} 0 & 1 \\ \frac{1}{2} & -\frac{1}{2} \end{pmatrix}$, therefore $$DF^2(x,y) = \begin{pmatrix} 0 & 1 \\ \frac{1}{2} & -\frac{1}{2} \end{pmatrix} \cdot \begin{pmatrix} 0 & 1 \\ \frac{1}{2} & -\frac{1}{2} \end{pmatrix} = \begin{pmatrix} \frac{1}{2} & -\frac{1}{2} \\ -\frac{1}{4} & \frac{3}{4} \end{pmatrix};$$ \noindent the corresponding eigenvalues are computed by the characteristic equation $$\lambda^2 - \frac{5}{4}\lambda + \frac{1}{4} = 0,$$ \noindent whose roots are given by $\lambda_1=1$ and $\lambda_2 =\frac{1}{4}$. Since one of the eigenvalues of the $2$-periodic point $(x,2-x)$ has modulo $1$, it is necessary to analyze its stability by other means different to the linearization technique.

As an initial example, consider $(x_{-1},x_0)=(0,0)$. We try to guess the behaviour at large of its associated orbit $(x_{-1}, x_0, x_1, x_2, \ldots, x_n, \ldots)$. The first elements are given by $$\left(0,0,1, \frac{1}{2}, \frac{5}{4}, \frac{5}{8}, \frac{21}{16}, \frac{21}{32},\ldots\right).$$

It is a simpler matter to see that all the elements are positive because by induction it is easily seen that $x_n = \frac{1}{3}(-1)^{n+1}- \frac{4}{3}\left(\frac{1}{2}\right)^{n+1}+1$ for all $n\geq -1$. Notice that this expression is nothing but the expression for the solution of the non-homogenous linear difference equation $x_{n+1} = 1 - \frac{1}{2}x_n + \frac{1}{2}x_{n-1}$, when initial conditions $(x_{-1},x_0) = (0,0)$ are considered. Therefore, $\lim_{n\rightarrow + \infty} x_{2n} = \frac{2}{3}$, $\lim_{n\rightarrow +\infty} x_{2n+1} = \frac{4}{3}$, and the orbit is asymptotically periodic to the $2$-periodic orbit $\left(\frac{4}{3}, \frac{2}{3}, \frac{4}{3}, \frac{2}{3}, \ldots \right)$.

The above computations for the orbit of $(0,0)$ suggest to do a case study, by distinguishing the different zones in which the initial conditions can be located.

Returning to our problem on the convergence of the orbits generated under $x_{n+1}=1-\frac{1}{2}|x_n| + \frac{1}{2}x_{n-1}$ in an asymptotically periodic form, we present the following result on invariance of two triangles in the plane by $F(x,y) = \left(y, 1 - \frac{1}{2}|y| + \frac{1}{2} x \right)$.

\begin{lemma}
Consider the triangles $$\Delta_{\ell} = \{(x,y): \ 0\leq x,y \leq  2, \ x+y \leq 2 \} $$ and $$\Delta_u = \{(x,y): \ 0\leq x,y \leq  2, \ x+y \geq 2 \}.$$ Then, $F(\Delta_{\ell}) \subset \Delta_{\ell}$ and $F(\Delta_u) \subset \Delta_u$, where $F(x,y)=(y,1-\frac{1}{2}|y|+\frac{1}{2}x)$.
\end{lemma}

\begin{proof}
Let $(x,y) \in \Delta_{\ell}$. Then, $F(x,y) = \left(y, 1 - \frac{1}{2}y + \frac{1}{2} x \right)$, with $0 \leq y \leq 2$, $$1-\frac{1}{2}y+\frac{1}{2}x \geq 1 - \frac{1}{2}y \geq 1 - 1 =0,$$ $$1 - \frac{1}{2}y + \frac{1}{2}x \leq 1 - \frac{1}{2}y + 1 \leq 2 - \frac{1}{2}y \leq 2, $$ and $$y + \left(1-\frac{1}{2}y+\frac{1}{2}x \right) = 1 + \frac{1}{2}x + \frac{1}{2}y \leq 1 + \frac{x+y}{2} \leq 1 + 1 = 2.$$ Therefore, $F(x,y) \in \Delta_{\ell}$.

The proof of the invariance of $\Delta_u$ by $F$ is analogous; now, $$ y + \left(1-\frac{1}{2}y+\frac{1}{2}x\right) = 1 + \frac{x+y}{2} \geq 2.$$
\end{proof}

As a consequence of this lemma, if some iterate of the orbit by $F$ lies in $\Delta_{\ell}$ or $\Delta_u$, then the orbit will remain indefinitely in these triangles. Since the coordinates of the iterates in these zones are always positive, in fact the Lozi map becomes in the linear system $F(x,y) = \left(y, 1 - \frac{1}{2}y + \frac{1}{2}x \right)$ and the solution will be explicitly known.

\begin{proposition}
Let $(x,y) \in \Delta_{\ell} \cup \Delta_u$. Then, $\left(F^n(x,y)\right)_{n\geq 0}$ is asymptotically periodic to some $2$-periodic point $(v,w)$.
\end{proposition}

\begin{proof}
Due to the invariance of $\Delta_{\ell} \cup \Delta_u$, if we denote the $n$-th iterate of $(x,y)$ under $F$ by $F^n(x,y) = (x_n,y_n), n\geq 0$, it is clear that $y_{n+1}= 1 - \frac{1}{2}y_n + \frac{1}{2}y_{n-1}$, with $y_{-1} = y$, $y_0 = 1 - \frac{1}{2}y + \frac{1}{2}x$. The general solution of the non-homogenous linear difference equation is given by $y_n = A(-1)^n + B\left(\frac{1}{2}\right)^n+1$, for arbitrary constants by $A,B \in \mathbb{R}$.

If we impose the initial conditions $y_{-1} = y$, $y_0 = 1 - \frac{1}{2}y + \frac{1}{2}x$, we obtain the solution $$y_n = \frac{2y-x-1}{3}(-1)^n + \frac{x+y-2}{3}\left(\frac{1}{2}\right)^n  + 1.$$ Since $y_{2n} \rightarrow \frac{2y-x-1}{3} + 1 = \frac{2y-x+2}{3}$ and $y_{2n+1} \rightarrow -\frac{2y-x-1}{3} + 1 = \frac{-2y+x+4}{3}$, we conclude that the orbit converge to the $2$-periodic sequence $$\left(\ldots, \frac{2y-x+2}{3}, \frac{-2y+x+4}{3}, \frac{2y-x+2}{3}, \frac{-2y+x+4}{3}, \ldots \right).$$
\end{proof}

Bearing in mind this proposition, in order to prove that all the orbits converge to periodic solutions of period $2$, it suffices to prove that for any initial condition $(x,y)$ there exists a positive integer $N$ such that $F^N(x,y) \in \Delta_{\ell} \cup \Delta_u$.

For this task, we are going to prove that for any square $$C_{m,n} := [2m,2m+2] \times [2n,2n+2], \ \ \text{for} \ m,n \in \mathbb{Z},$$ \noindent either there exists $N\in \mathbb{Z}$ such that $F^N(C_{m,n}) \subset C_{0,0}$ or $\cap_{p\geq 0} F^p(C_{m,n}) = \{(0,2),(0,2)\}$. In both cases, we will deduce that for any $(x,y) \in \mathbb{R}^2$, it holds that $(F^n(x,y))_{n\geq 0}$ is asymptotically periodic to some $2$-periodic point $(v,w)$.

To do it, we will use the following facts. In their proofs, we will use the notation $$H_u = \{(x,y):y\geq 0\} \ \ \text{and} \ \ H_{\ell}=\{(x,y):y\leq 0\}.$$

\begin{lemma}
The map $F$ transforms the square $C_{m,n}$ into a parallelogram and the vertices of $C_{m,n}$ go to the vertices of that parallelogram.
\end{lemma}

\begin{proof}
It is a consequence of the linearity of $F$ and the fact of being located the square $C_{m,n}$ entirely on the upper half-plane $H_u$ or in the lower half plane $H_{\ell}$.
\end{proof}

\begin{lemma}
Let $A(C_{m,n})$ and $A(F(C_{m,n}))$ be the areas of the square $C_{m,n}$ and its image by $F$, respectively. Then $A\left(F(C_{m,n})\right) = \frac{1}{2} A(C_{m,n})$.
\end{lemma}

\begin{proof}
In this case,  $A\left(F(C_{m,n})\right) = \int \int_{C_{m,n}} \left|\frac{\partial(x,y) }{\partial(u,v)}\right| du \ dv$, where $\left|\frac{\partial(x,y)}{\partial(u,v)}\right|$ is the absolute value of the determinant $\left|\begin{matrix} \frac{\partial x}{\partial u} & \frac{\partial x}{\partial v} \\ \frac{\partial y}{\partial u} & \frac{\partial y}{\partial v} \end{matrix} \right|$, with $(x(u,v), y(u,v)) = \left(v, 1 - \frac{1}{2}v + \frac{1}{2}u \right)$ or $(x(u,v), y(u,v)) = \left(v, 1 + \frac{1}{2}v + \frac{1}{2}u \right)$ depending on whether the square is located or not in the upper half-plane $H_u$. In both cases, $\frac{\partial(x,y)}{\partial(u,v)} = \left|\begin{matrix} \frac{\partial x}{\partial u} & \frac{\partial x}{\partial v} \\ \frac{\partial y}{\partial u} & \frac{\partial y}{\partial v} \end{matrix} \right| = \left|\begin{matrix} 0 & 1 \\ \frac{1}{2} & \pm \frac{1}{2} \end{matrix} \right| = -\frac{1}{2}$, and the result follows.
\end{proof}

Concerning the iteration of triangles located entirely on $H_u$ or $H_{\ell}$, again we have that any triangle $T \subset H_u$  (respectively, $T\subset H_{\ell}$) is transformed in a new triangle whose vertices are the vertices of $T$ by the action of $F$, and additionally $A(F(T)) = \frac{1}{2}A(T)$.

From now on, our strategy consists in proving:

\begin{itemize}
    \item[\textbf{(a)}] The region $R_1$ surrounding $C_{0,0}$, including $C_{0,0}$ itself, is invariant, that is, $R_1 := \bigcup_{i,j\in\{-1,0,1\}} C_{i,j}=[-2,4]\times [-2,4]$ is invariant; after that, to prove that either the images $F^N(C_{i,j})$ are included in $C_{0,0}$ for some positive integer $N$ or the intersection of images not included in $C_{0,0}$ converges to the $2$-periodic orbit $\{(0,2), (0,2)\}$.
    \item[\textbf{(b)}] The image of any square $C_{i,j}$ for $i \geq 2$ or $j\geq 2$ is finally entirely contained in $R_1$. 
\end{itemize}

For \textbf{(a)}, we need to use the previous lemmas concerning the contraction of areas and the transformation of triangles and parallelograms in the same type of figures whenever the geometric object is included in $H_{\ell}$ or $H_u$.

As a first step, we analyze the evolution of the segment $$S_{\varepsilon,0} = \{(x,0):2\leq x \leq 2+\varepsilon \}, \,\,\varepsilon \in [0,1].$$

\begin{lemma} \label{Lemma_S1}
 Given $\varepsilon \in [0,1]$, consider $F^n(S_{\varepsilon,0}), n \geq 0$. Then, there exists $N$ such that $F^N(S_{\varepsilon}) \subset C_{0,0}$.
\end{lemma}

\begin{proof}
It suffices to study the evolution of the endpoints of $S_{\varepsilon,0}$. Realize that $F^n(S_{\varepsilon,0})$ are lines lying entirely in either $H_{\ell}$ or $H_u$. Then, $$(2+\varepsilon, 0) \xrightarrow{F} \left(0,2+\frac{\varepsilon}{2}\right) \xrightarrow{F} \left( 2 + \frac{\varepsilon}{2}, -\frac{\varepsilon}{4} \right) \xrightarrow{F} \left(-\frac{\varepsilon}{4}, 2 + \frac{\varepsilon}{8}\right) \xrightarrow{F} \left(2 + \frac{\varepsilon}{8}, -\frac{3}{16}\varepsilon \right) $$ $$\xrightarrow{F} \left(-\frac{3}{16}\varepsilon, 2 - \frac{\varepsilon}{32}\right) \xrightarrow{F} \left( 2 -\frac{\varepsilon}{2}, - \frac{5}{64}\varepsilon \right) \xrightarrow{F} \left(- \frac{5}{64}\varepsilon , 2- \frac{7}{128}\varepsilon  \right) $$ $$\xrightarrow{F} \left(2- \frac{7}{128}\varepsilon, -\frac{3}{256}\varepsilon \right) \xrightarrow{F} \left(-\frac{3}{256}\varepsilon, 2 - \frac{17}{512}\varepsilon \right) \xrightarrow{F} \left(2 - \frac{17}{512}\varepsilon, \frac{11}{1024}\varepsilon \right). $$

At this point, we notice that $F^{10}(2+\varepsilon,0) = \left(2 - \frac{17}{512}\varepsilon, \frac{11}{1024}\varepsilon \right) \in C_{0,0}$, which ends the proof.
\end{proof}

In a similar way, we will keep studying the behavior of the lines $$S_{\varepsilon,2} = \{(x,2): -\varepsilon \leq x \leq 0 \}, \,\,\varepsilon \in [0,1].$$

\begin{lemma} \label{Lemma_S2}
 Given $\varepsilon \in [0,1]$, consider $F^n(S_{\varepsilon,2}), n\geq 0$. Then, there exists $N$ such that $F^N(S_{\varepsilon,2}) \subset C_{0,0}$.
\end{lemma}

\begin{proof}
 In this case, the iterates of the endpoint $(-\varepsilon,2)$ are $$(-\varepsilon, 2) \xrightarrow{F} \left(2, -\frac{\varepsilon}{2}\right) \xrightarrow{F} \left( -\frac{\varepsilon}{2}, 2 - \frac{\varepsilon}{4}\right)\xrightarrow{F} \left(2 - \frac{\varepsilon}{4}, -\frac{\varepsilon}{8} \right) $$ $$\xrightarrow{F} \left(-\frac{\varepsilon}{8}, 2 - \frac{3}{16}\varepsilon \right) \xrightarrow{F} \left(2 - \frac{3}{16}\varepsilon, \frac{\varepsilon}{32} \right).$$ Since $\left(2 - \frac{3}{16}\varepsilon, \frac{\varepsilon}{32} \right) \in C_{0,0}$, we finish the proof.
\end{proof}

As a second step, we are going to analyze the behavior of $F$ under appropriate neighbourhoods of $(2,0)$ and $(0,2)$. The proof is straightforward and we omit it. Instead, we encourage the reader to follow the reasoning via Figure \ref{Figure_a0.5}.

\begin{figure}[ht] 
\includegraphics[scale=0.5]{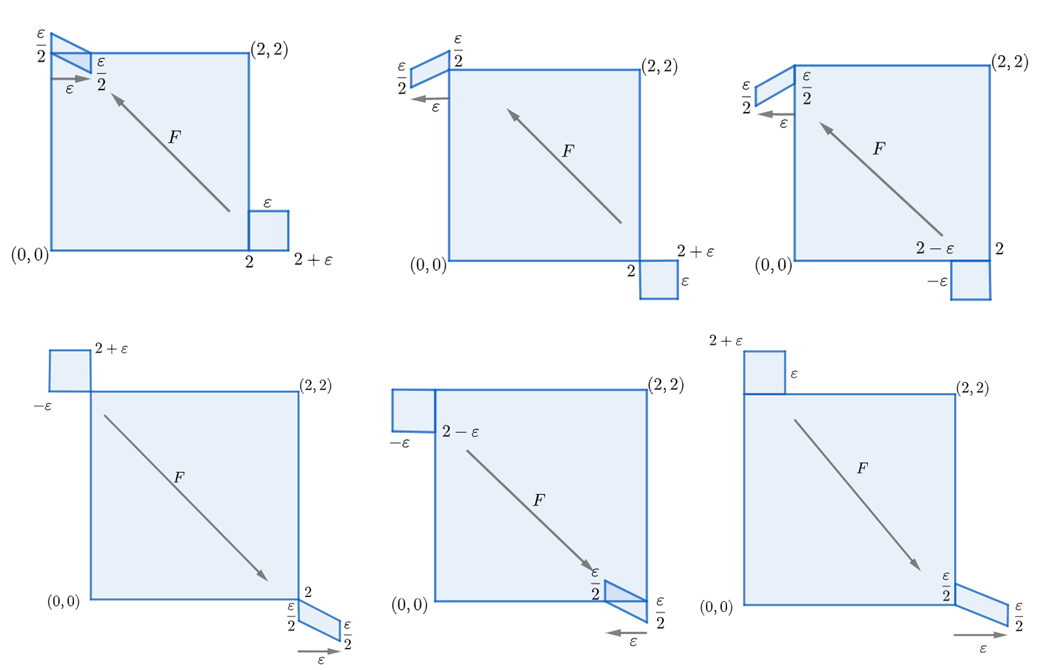}
\caption{Evolution of neighbourhoods $Q_{j,\varepsilon}$ of $(2,0)$ (top) and neighbourhoods $W_{j,\varepsilon}$$(0,2)$ (bottom), $j\in\{1,2,3\}$, in Lemma~\ref{Lemma_S3}. }
\label{Figure_a0.5}
\end{figure}

\begin{lemma} \label{Lemma_S3}
 Let $\varepsilon \in [0,1]$. Consider the squares $Q_{1,\varepsilon} = [2,2+\varepsilon] \times [0,\varepsilon]$, $Q_{2,\varepsilon} = [2,2+\varepsilon] \times [-\varepsilon,0]$ and $Q_{3,\varepsilon} = [2-\varepsilon,2] \times [-\varepsilon,0]$. Then:
\begin{itemize}
    \item[\textbf{(a.1)}] $F(Q_{1,\varepsilon}) \subset C_{0,0}\, \cup\, T_{1,\varepsilon}$, where $T_{1,\varepsilon}$ is the triangle with vertices $(0,2)$, $\left(0, 2 +\frac{\varepsilon}{2}\right)$, $(\varepsilon, 2)$.
    \item[\textbf{(a.2)}] $F(Q_{2,\varepsilon})$ is the parallelogram with vertices $(0,2), \left(0, 2 +\frac{\varepsilon}{2}\right)$, $(-\varepsilon, 2)$, $\left(-\varepsilon, 2 - \frac{\varepsilon}{2} \right)$.
    \item[\textbf{(a.3)}] $F(Q_{3,\varepsilon})$ is the parallelogram with vertices $(0,2), \left(0, 2 -\frac{\varepsilon}{2}\right)$, $(-\varepsilon, 2 -\frac{\varepsilon}{2})$, $\left(-\varepsilon, 2 - \varepsilon \right)$.
\end{itemize}
Consider the squares $W_{1,\varepsilon} = [-\varepsilon,0] \times [2,2+\varepsilon]$, $W_{2,\varepsilon} = [-\varepsilon,0] \times [2-\varepsilon, 2]$ and $W_{3,\varepsilon} = [0,\varepsilon] \times [2,2+\varepsilon]$. Then:
\begin{itemize}
    \item[\textbf{(b.1)}] $F(W_{1,\varepsilon})$ is the parallelogram with vertices $(2,0),$ $\left(2,-\frac{\varepsilon}{2}\right)$, $(2+\varepsilon, - \varepsilon)$, $\left(2+\varepsilon, - \frac{\varepsilon}{2} \right)$.
    \item[\textbf{(b.2)}] $F(W_{2,\varepsilon}) \subset C_{0,0}\, \cup\, T_{2,\varepsilon}$, where $T_{2,\varepsilon}$ is the triangle with vertices $(2-\varepsilon, 0)$, $(2,0)$, $\left(2, -\frac{\varepsilon}{2} \right)$.
    \item[\textbf{(b.3)}] $F(W_{3,\varepsilon})$ is the parallelogram with vertices $(2,0),$ $\left(2,\frac{\varepsilon}{2}\right)$, $(2+\varepsilon, 0)$, $\left(2+\varepsilon, - \frac{\varepsilon}{2} \right)$.
\end{itemize}
\end{lemma}

We are now in a position to prove that all the orbits of the region $R_1 = [-2,4] \times [-2,4]$ enter the square $C_{0,0}$.

\begin{proposition}\label{P:inicio}
For any point $(x,y)$ in the region $R_1$, its orbit $(F^n(x,y))_n$ eventually enters in the square $C_{0,0}$.
\end{proposition}

\begin{proof}
Taking into account that $R_1 = \bigcup_{-1\leq i,j \leq 1} C_{i,j}$, we proceed by cases. We only present the proof for some of them, the rest is left to the care of the reader.
 
If $(x,y) \in C_{1,1}$, we consider that $F(C_{1,1}) \subset C_{1,0}$ is the parallelogram having vertices $(2,1)$, $(2,2)$, $(4,1)$, $(4,0)$, and the second iterate $F^2(C_{1,1})$ is the parallelogram with vertices $\left(1,\frac{5}{2}\right)$, $(2,1)$, $\left(1,\frac{3}{2}\right)$, $(0,3)$. Notice that $F^2(C_{1,1}) \subset C_{0,0} \cup W_{3,\varepsilon} \cup \tilde{T}$ with $\varepsilon=1$ and $\tilde{T}$ the triangle of vertices $\left(1,\frac{5}{2}\right)$, $\left(\frac{4}{3},2\right)$ and $(1,2)$; moreover, $F^2(\tilde{T}) \subset C_{0,0} \cup W_{3,\varepsilon}$, since the vertices of $\tilde{T}$ by $F^2$ go to $\left(\frac{1}{4}, \frac{17}{8}\right)$, $\left( \frac{2}{3}, \frac{5}{3} \right)$ and $\left(\frac{1}{2}, \frac{7}{4} \right)$. From here, by Lemmas \ref{Lemma_S1}-\ref{Lemma_S3}, we deduce that the images of $W_{3,\varepsilon}$ by $F$ either enter to $C_{0,0}$ or either the intersection of such images narrow to segments of type $S_{\varepsilon,0}$ or $S_{\varepsilon,2}$. In any case, we derive the statement of the proposition.
 
For the square $C_{1,0}$, we have that its image $F(C_{1,0})$ is the parallelogram with vertices $(0,2)$, $(0,3)$, $(2,2)$, $(2,1)$ and consequently $F^2(C_{1,0})$ is a new parallelogram having vertices $(2,0)$, $\left(3,-\frac{1}{2}\right)$, $(2,1)$, $\left(1,\frac{3}{2}\right)$; since $F^2(C_{1,0}) \subset C_{0,0} \cup Q_{1,\varepsilon} \cup Q_{2,\varepsilon}$, with $\varepsilon=1$, a similar reasoning to that carried out in the former case gives us the desired conclusion on the enter of the orbit into $C_{0,0}$.
 
For the square $C_{-1,-1} = [-2,0] \times [-2,0]$, we find that $F(C_{-1,-1})$ is a parallelogram with vertices $(-2,-1)$, $(-2,0)$, $(0,1)$, $(0,0)$; we decompose the figure into two triangles $T_1 \cup T_2$, with vertices $(-2,-1)$, $(-2,0)$, $(0,0)$ for $T_1$ and $(-2,0)$, $(0,0)$, $(0,1)$ for $T_2$; then $F(T_2) \subset C_{0,0}$ while $F(T_1)$ is sent to the new triangle $T_3$ with vertices $\left(-1,-\frac{1}{2}\right)$, $(0,0)$, $(0,1)$. Since $T_3$ intersects both half-planes $H_{\ell}$ and $H_u$, we need to decompose it as $T_4 \cup T_5$; in this case, consider that $T_4$ has vertices $\left(-1,-\frac{1}{2}\right)$, $\left(-\frac{2}{3},0\right)$, $(0,0)$, so $F(T_4)$ has vertices $\left(-\frac{1}{2},\frac{1}{4}\right)$, $\left(0,\frac{2}{3}\right)$, $(0,1)$ and it is easily seen that, in fact, $F^2(T_4) \subset C_{0,0}$; with respect to the triangle $T_5$, with vertices $\left(-\frac{2}{3},0\right)$, $(0,0)$, $(0,1)$, it is a simple matter to see that $F(T_5) \subset C_{0,0}$.
\end{proof}

For \textbf{(b)} we need to control the evolution of the images corresponding to the squares $C_{i,j}$. We want to generalize the process by induction in the different levels $2, 3, 4,\ldots$ To this purpose, recall that the vertices of the squares are sent to the vertices of the new parallelograms.

\begin{figure}[ht] 
\includegraphics[scale=0.5]{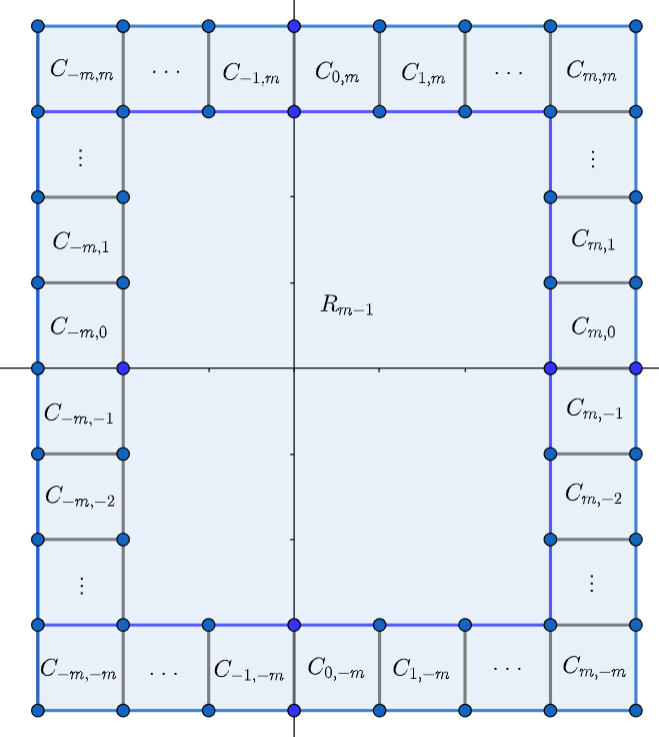}
\caption{The squares $C_{m,n}$, $C_{n,m}$, $C_{-m,n}$ and $C_{n,-m}$ with $n\in[-m,m]$.}
\label{Cmn}
\end{figure}

Suppose that $m\geq 2$. 
If we denote $R_t=\cup_{-t\leq i,j \leq t} C_{i,j}$, we assume that the orbits of points of $R_t$ by $F$ eventually enter into the region $R_{t-1}$, $1\leq t \leq m-1$, and consequently these orbits converge to $2$-periodic points of the square $C_{0,0}$.

In order to apply the process of mathematical induction, we start by studying the movement of squares $C_{m-j,m}$ and $C_{m,m-j}$, $j=0,1,\ldots, m$, as follows (for the proof, it suffices to evaluate the images of the vertices of the squares):

\begin{lemma} \label{LemmaA}
 Let $m\geq 2$ be a positive integer. For any value $j\in \{0,1,\ldots,m\}$:
 
\begin{itemize}
    \item[(a)] $F(C_{m-j,m}) \subset \left\{ \begin{matrix} C_{m,-\frac{j}{2}} & \text{if} \ j \ \text{is even}, \\ C_{m, \frac{-j-1}{2} } \cup C_{m, \frac{-j+1}{2}} & \text{if} \ j \ \text{is odd}. \end{matrix} \right.$
    \item[(b)] $F(C_{m,j}) \subset \left\{ \begin{matrix} C_{j,\frac{m-j}{2}} & \text{if} \ m-j \ \text{is even}, \\ C_{j, \frac{m-j-1}{2}} \cup C_{j, \frac{m-j+1}{2}} & \text{if} \ m-j \ \text{is odd}. \end{matrix} \right.$
\end{itemize}
\end{lemma}

\begin{proof} (a) Notice that the vertices of $C_{m-j,m}$ are sent to:
$$F(2m-2j, 2m) = (2m,-j+1), \ F(2m-2j, 2m+2) = (2m+2,-j),$$ $$F(2m-2j+2, 2m+2) = (2m+2,-j+1), \ F(2m-2j+2,2m) = (2m,-j+2).$$

From here, it follows part (a). The proof of part (b) is similar and we omit it.
\end{proof}
The following result, whose proof is direct, gives us the evolution of $C_{-j,m}, j=0,1,\ldots,m.$
\begin{lemma}\label{LemmaA-B}
Let $m\geq 2$ be a positive integer. For any value $j\in \{0,1,\ldots,m\}$:
 
$$F(C_{-j,m}) \subset \left\{ \begin{matrix} C_{m,\frac{-m-j}{2}} & \text{if} \ m+j \ \text{is even}, 
		\\ C_{m,\frac{-m-j-1}{2}} \cup C_{m,\frac{-m-j+1}{2}} & \text{if} \ m+j \ \text{is odd}. \end{matrix} \right.$$
\end{lemma}

With respect to the iterates of $C_{m,-j}$, $0\leq j \leq m$, we obtain the following lemma whose proof is direct.

\begin{lemma} \label{LemmaB}
 Let $m\geq 2$ be a positive integer. For any value $j\in \{0,1,\ldots,m\}$: $$F(C_{m,-j}) \subset \left\{ \begin{matrix} C_{-j,\frac{m-j+1}{2}} & \text{if} \ m-j \ \text{is odd}, \\ C_{-j, \frac{m-j}{2}} \cup C_{-j,\frac{m-j+2}{2}} & \text{if} \ m-j \ \text{is even}. \end{matrix} \right. $$
\end{lemma}

It remains to analyze the iterates of $C_{j,-m}$ and $C_{-m,j}$ for $-m\leq j \leq m$.

\begin{lemma} \label{LemmaC}
 Let $m\geq 2$ be a positive integer. For any integer value $j\in [-m,m]$: $$F(C_{j,-m}) \subset \left\{ \begin{matrix} C_{-m,\frac{-m+j+1}{2}} & \text{if} \ m-j \ \text{is odd}, \\ C_{-m, \frac{-m+j}{2}} \cup C_{-m,\frac{-m+j+2}{2}}, & \text{if} \ m-j \ \text{is even}. \end{matrix} \right. $$
\end{lemma}

We stress that, for $j=-m$, we have $F(C_{-m,-m})\subset C_{-m,-m} \cup C_{-m,-m+1}$. This means that this case has to be discussed carefully in order to obtain that its iterates enter to $R_{m-1}$ at some moment. This will be clarified with the images $F(C_{m,-j})$ and with the tracing of the parallelograms obtained in successive steps.

Before, let us state the evolution of the images of $C_{-m,j}$ for integer values $j\in[-m,m]$. Again, we omit the proofs.

\begin{lemma} \label{LemmaD}
 Let $m\geq 2$ be a positive integer. For any value $j\in \{0,1,\ldots,m\}$: $$F(C_{-m,j}) \subset \left\{ \begin{matrix} C_{j,\frac{-m-j}{2}} & \text{if} \ m+j \ \text{is even}, \\ C_{j, \frac{-m-j-1}{2}} \cup C_{j,\frac{-m-j+1}{2}}, & \text{if} \ m+j \ \text{is odd}. \end{matrix} \right. $$
\end{lemma}

\begin{lemma} \label{LemmaE}
 Let $m\geq 2$ be a positive integer. For any value $j\in \{0,1,\ldots,m\}$: $$F(C_{-m,-j}) \subset \left\{ \begin{matrix} C_{-j,\frac{-m-j+1}{2}} & \text{if} \ m+j \ \text{is odd}, \\ C_{-j, \frac{-m-j}{2}} \cup C_{-j,\frac{-m-j+2}{2}}, & \text{if} \ m+j \ \text{is even}. \end{matrix} \right. $$
\end{lemma}

Finally, the promised result concerning the evolution of $C_{-m,-m}$:

\begin{lemma} \label{LemmaF} Let $m\geq 2$ be an arbitrary integer.
 It holds $F(C_{-m,-m}) \subset C_{-m,-m} \cup C_{-m,-m+1}$. In particular:

(a) $F(C_{-m,-m})$ is the parallelogram with vertices $(-2m+2, -2m+3)$, $(-2m+2,-2m+2)$, $(-2m,-2m+1)$, $(-2m,-2m+2)$;

(b) $F^2(C_{-m,-m})$ is the parallelogram with vertices $\left(-2m+1,-2m+\frac{3}{2}\right)$, $(-2m+2, -2m+2)$, $\left(-2m+3, -2m+\frac{7}{2}\right)$, $(-2m+2,-2m+3)$;

(c) $F^2(C_{-m,-m}) \subset C_{-m,-m} \cup \, C_{-m,-m+1} \cup \, C_{-m+1,-m+1}$, and the part contained in $C_{-m,-m}\cup \, C_{-m,-m+1}$ is the triangle $T_m$ with vertices $\left(-2m+1, -2m+\frac{3}{2} \right)$, $(-2m+2, -2m+2)$, $(-2m+2, -2m+3)$;

(d) $F(T_m)$ is the triangle with vertices $\left(-2m+\frac{3}{2}, -2m+\frac{9}{4} \right)$, $(-2m+2, -2m+3)$, $\left(-2m+3, -2m+\frac{7}{2} \right)$. Therefore, $F(T_m) \subset C_{-m,-m+1} \cup R_{m-1}$.
\end{lemma}

We collect the previous study in the following result.

\begin{proposition}
Let $$(x,y)\in \bigcup_{-m\leq j \leq m}(C_{j,m}\cup C_{m,j} \cup C_{j,-m} \cup C_{-m,j}),$$ with $m\geq 1$. Then, there exists a positive integer $N$ such that $F^N(x,y) \in R_{m-1}$.
\end{proposition}

\begin{proof}
The proof follows directly from Proposition~\ref{P:inicio} and Lemmas \ref{LemmaA}-\ref{LemmaF}.
\end{proof}

As a consequence of all our study, we conclude:

\begin{theorem} \label{Th_a=0.5}
Given the difference equation $$x_{n+1} = 1 - \frac{1}{2}|x_n| + \frac{1}{2}x_{n-1}, $$ its dynamics is given by:
\begin{itemize}
    \item[\textbf{a)}] An equilibrium point,  $\overline{x}=1$.
    \item[\textbf{b)}] A continuum of $2$-periodic sequences $(\ldots,x,y,x,y,\ldots)$ with $0\leq x,y \leq 2$, $x+y=2$.
    \item[\textbf{c)}] The rest of solutions converge to one of the $2$-periodic solutions given in Part (b).
\end{itemize}
\end{theorem}

\begin{remark}
It should be noticed that in \cite{Botella11}, the authors detect the continuum of $2$-periodic sequences for arbitrary $a$ and $b$ verifying the constraints $a>0$ and either $a+b=1$ or $b-a=1$. In fact, they plot, fixing $b=0.5$, a superposition of twenty attractors for the values $a = 0.1\cdot n$, with $n=1, \ldots, 20$. In particular, for $a=0.5$, they show a segment line representing the attractor. In our work we have proved analytically this property of global attraction. It may be of some interest to study such problem for other values of $a$ and $b$.
\end{remark}

As a consequence of our study on the relationship between the Lozi map and suitable max-type equations (realize that $\delta = 1$, $\alpha + \beta + \gamma \neq 1$ and use Proposition \ref{Prop_deltaneq0}), we deduce the following result.

\begin{corollary}
Any element of the family of difference equations $$x_{n+1} = (x_n \cdot x_{n-1})^{\frac{1}{2}} \cdot \max \left\{ \frac{1}{x_n}, A \right\},$$ defined for any arbitrary positive real initial conditions, with $0<A<1$, presents the following dynamics: there exists a unique equilibrium point, $\bar{x}=1$; there exists a continuum of $2$-periodic solutions constructed from the initial conditions $\left(\ldots, x, \frac{1}{x}, x, \frac{1}{x}, \ldots \right)$, with $x \in \left[A, \frac{1}{A}\right]$; the rest of solutions converge to a $2$-periodic solution.
\end{corollary}

\subsection{Case $a=-\frac{1}{2}$}

We consider the parameter $a=-\frac{1}{2}$ for the Lozi map
\begin{equation} \label{Eq_Lozi_a-0.5}
    x_{n+1} = 1 + \frac{1}{2}|x_n| - \frac{1}{2}x_{n-1}.
\end{equation}

In this case, we are going to prove that the unique equilibrium point of Equation (\ref{Eq_Lozi_a-0.5}), namely $\bar{x}=1$, is in fact a global attractor. The strategy is strongly similar to that developed in the case $a=\frac{1}{2}$, so we will only outline the proof.

The bidimensional map associated to the difference equation is given by $F_{-\frac{1}{2}}(x,y) = \left(y, 1 + \frac{1}{2}|y| - \frac{1}{2}x \right)$, and $DF_{-\frac{1}{2}}(1,1) = \begin{pmatrix} 0 & 1 \\ -\frac{1}{2} & \frac{1}{2} \end{pmatrix}$; since the eigenvalues of this matrix are $\frac{1}{4} \pm \frac{i\sqrt{7}}{4}$, having a modulo less than $1$, we deduce that, at least, the equilibrium point is locally asymptotically stable\footnote{Roughly speaking, we say that an equilibrium point $\bar{x}$ is locally asymptotically stable if $\bar{x}$ is locally stable and if in addition is locally attractor. For a precise definition, consult \cite{Elaydi05}.}.

We maintain the notation employed in the analysis of the case $a=\frac{1}{2}$, and write $C_{m,n} = [2m, 2m+2] \times [2n, 2n+2]$. It is a simple matter to check (see Figure \ref{Diagramafig2}) that the square $C_{0,0}$ is invariant by $F_{-\frac{1}{2}}$; this implies that we move in the upper half-plane $H_u$, and consequently the dynamics of the Lozi map $F_{-\frac{1}{2}}$ is governed by the linear difference equation $x_{n+1} = 1 + \frac{1}{2}x_n - \frac{1}{2}x_{n-1}$. Given two initial conditions $(x_{-1},x_0) = (x,y) \in C_{0,0}$, its corresponding solution is \small$$x_n = \left(\frac{1}{2}\right)^{\frac{n+1}{2}} \cdot \left((x-1)\cos ((n+1)\theta) + \left[\frac{(y-1)\sqrt{2}}{\sin (\theta)} + (1-x)\cot (\theta)\right] \sin ((n+1)\theta) \right) + 1,$$
for $n \geq -1,$ as can be easily checked (here, $\theta=\arctan(\sqrt{7})$). Then, $$ \lim_{n\rightarrow + \infty} x_n = 1 = \bar{x}.$$ \normalsize

According to the last result, in order to prove that $\bar{x}=1$ is a global attractor, it suffices to show that any square $C_{m,n}$ is eventually sent to $C_{0,0}$. For instance, $F_{-\frac{1}{2}}(C_{0,-1}) \subset C_{-1,0}$, and $F^3_{-\frac{1}{2}}(C_{-1,0}) \subset C_{0,0}$. The reasoning is completely analogous to that of case $a=\frac{1}{2}$, it is necessary to see that is always possible to descend from a level $R_m = \bigcup_{-m\leq i,j \leq m} C_{i,j}$ to the precedent level $R_{m-1} = \bigcup_{-m+1\leq i,j \leq m-1} C_{i,j}$ and check that, in fact, the iterates of points of $R_1 = \bigcup_{-1\leq i,j \leq 1} C_{i,j}$ eventually enter in $C_{0,0}$. We leave the details to the care of the reader.

\begin{figure}[ht] 
\includegraphics[scale=0.4]{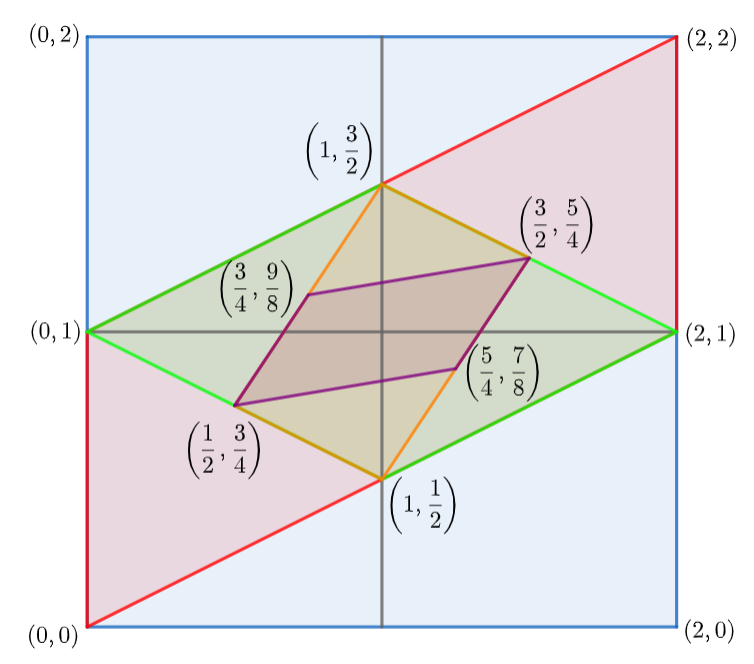}
\caption{Dynamics of Equation (\ref{Eq_Lozi_a-0.5}) in the square $[0,2]\times [0,2]$.}
\label{Diagramafig2}
\end{figure}

\section{Numerical simulations} \label{Sec_simulation}

As a consequence of the transformations developed in Section \ref{Sec_transformation}, the detection of a certain dynamical characteristic or property of a single max equation $x_{n+1} = \frac{\max\{x_n^k, A\} }{x_n^l x_{n-1}^m}$, for a concrete value $A>1$ or $0<A<1$, will guarantee that all the elements of the family, with $A>1$ or $0<A<1$, will share that property. In this sense, we have developed several numerical simulations in order to apply and illustrate the results proved in the previous sections. Moreover, we intend to progress in new open problems related to the topic. 

In concrete, in this section, firstly we will deal with an equation studied by Abu-Saris and Allan in \cite{AbuSarisA97}. They show the existence of a \textit{strange} attractor for a particular value of the parameter that is involved in the equation. Here, we will see that, in fact, the same behaviour appears for the whole uniparametric family of max-equations. It would be of interest to find other families of max-type equations exhibiting a strange attractor. 

Next, as a first step of this search of complicated behaviour, we will focus on the evolution of the origin, $(x_{-1},x_0)=(0,0)$, under Lozi map in the particular case $a=b$, one of the cases that allow us to connect with max-equations. In this regard, we will present some numerical simulations related to the particular case $a=b$ of Lozi map. At first sight, it seems that no complicated behaviours are attained. It should be highlighted that, taking into account the study developed in the previous section, it will also be interesting to analyzed the case $b=a+2$.

\subsection{A generic property}

In \cite{AbuSarisA97}, the authors show the figure of a \textit{strange} attractor when they consider the difference equation
\begin{equation}\label{max_AbuSaris}
    x_{n+1} = \frac{\max\{x_n^2,A\}}{x_n x_{n-1}},
\end{equation}

\noindent in the particular case $A=2.3$. Notice that Equation (\ref{max_AbuSaris}) is topologically conjugate to $$y_{n+1} = |y_n| - y_{n-1} - 1.$$

They also comment that \textit{the solution in this case is also chaotic for certain values of $A$}. In fact, according to our study, and due to the fact that topological conjugation is a transitive relation, we deduce that all the elements of the family, $x_{n+1} = \frac{\max\{x_n^2,B\}}{x_n x_{n-1}}$, with $B>1$, will present the \textit{same} strange attractor, a homeomorphic copy of the attractor detected for $A=2.3$. It is not difficult to reproduce the picture of Abu-Saris and Allan (see Figure \ref{Abu1}), as well as new figures showing us the same dynamics, see Figure \ref{Abu2}.

\begin{figure}[ht] 
\includegraphics[scale=0.5]{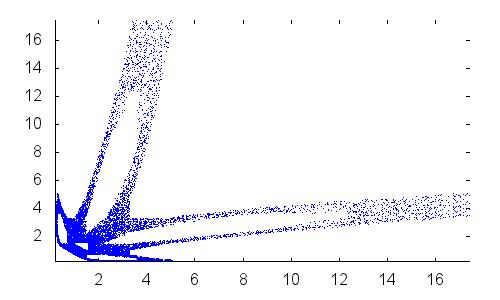}
\caption{Reproduction of Abu-Saris Allan's figure 4 in \cite{AbuSarisA97}.}
\label{Abu1}
\end{figure}

\begin{figure}[ht] 
\includegraphics[scale=0.5]{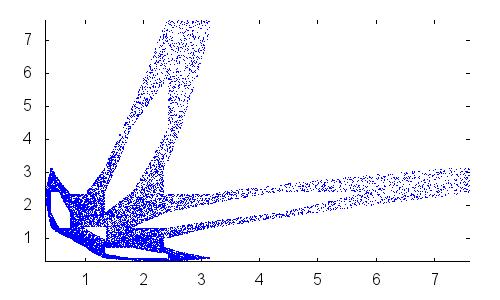}
\caption{Simulation with $A=1.8$ and initial conditions $x_{-1} = x_0 = 1.8^{0.8}$.}
\label{Abu2}
\end{figure}

For $A<1$, in order to do the transformation of a max-equation $x_{n+1} = \frac{\max\{x_n^2,A\}}{x_n x_{n-1}}$ into a generalized Lozi map, it must be $k-l-m-1=-\delta$, so in this case we conclude that it is topologically conjugate to $y_{n+1} = |y_n| - y_{n-1} + 1$. It seems that the same behavior than the one exhibits when $A>1$ takes place. See Figure \ref{Abu3}.

\begin{figure}[ht] 
\includegraphics[scale=0.5]{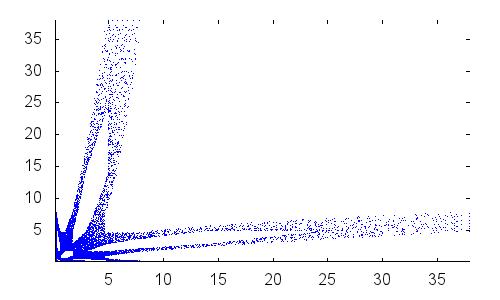}
\caption{Simulation with $A=0.35$ and two random initial conditions $x_{-1}, x_0 \in (0,1)$.}
\label{Abu3}
\end{figure}

In view of Figures~\ref{Abu2} and~\ref{Abu3}, a natural question arises, to know if the systems are topologically conjugate.

\subsection{Case $a=b$}

Now, we will simulate the evolution of the origin, $(x_{-1},x_0)=(0,0)$, for different values of the parameters $a$ and $b$ in the particular case $a=b$, in order to deduce some consequences for the associated family of max-types equations. Moreover, we will focus on the unknown cases $a\notin \left[-\frac{1}{2}, \frac{1}{2}\right]$. In the sequel,  we assume $(x_{-1},x_0)=(0,0)$.

Firstly, the case $a=b=-1$ yields to a $6$-periodic orbit. When $a=b<-1$, it seems that the orbit of $(0,0)$ tends to infinity in a spiral movement as Figure \ref{Fig2}  shows. It would be interesting to prove if any initial conditions verify this class of dynamics.

\begin{figure}[ht] 
\includegraphics[scale=0.5]{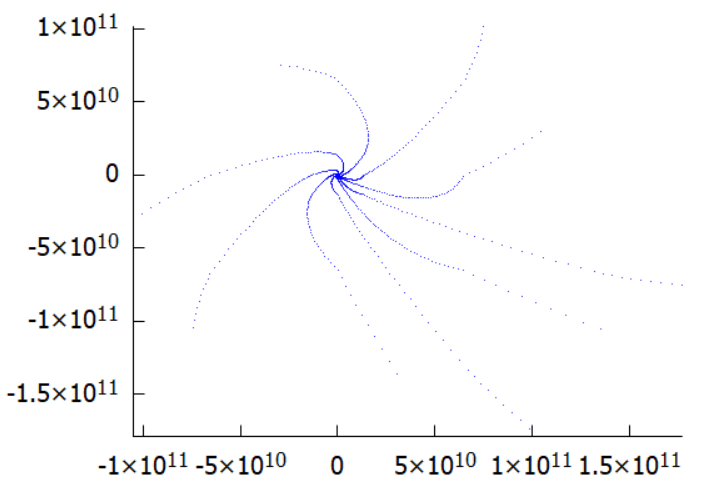}
\caption{Case $a=b=-1.01$. A repulsive orbit. }
\label{Fig2}
\end{figure}

Next, when $a\in \left(-1, -\frac{1}{2}\right)$, it seem that the orbit of $(0,0)$ is always trapped by an equilibrium point, which is different for distinct values of $a$. In this sense, it would be of interest to prove analytically whether or not the equilibrium point is a global attractor. Probably, a similar technique to that developed in Section \ref{S:caso_a=b} could be successful.

\begin{figure}[ht] 
\includegraphics[scale=0.6]{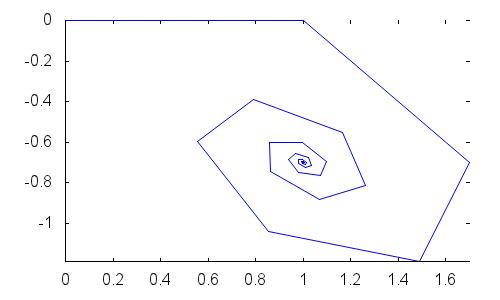}
\caption{Case $a=b=-0.7$. The orbit of $(0,0)$ goes to an equilibrium point.}
\label{Fig3}
\end{figure}

However, for values $a>\frac{1}{2}$, the situation changes drastically. Now, we have two unstable equilibrium points, namely $\bar{x}=1$ and $\bar{x}=\frac{1}{1-2a}$, see Lemma~\ref{L:points}. Furthermore, it seems that there exists a $2$-periodic orbit which is a local attractor that attracts the origin. We illustrate in Figures \ref{Fig6} and \ref{Fig7} the cases $a=b=0.99$ and $a=b=0.9998$.

\begin{figure}[ht] 
\includegraphics[scale=0.6]{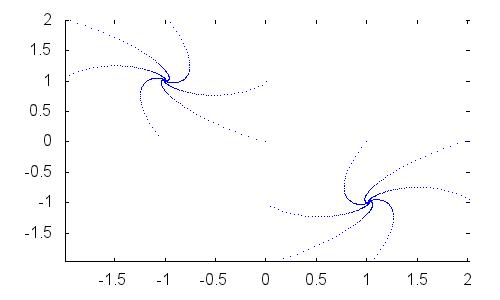}
\caption{Case $a=b=0.99$. The orbit of $(0,0)$ goes to a $2$-periodic orbit.}
\label{Fig6}
\end{figure}

\begin{figure}[ht] 
\includegraphics[scale=0.7]{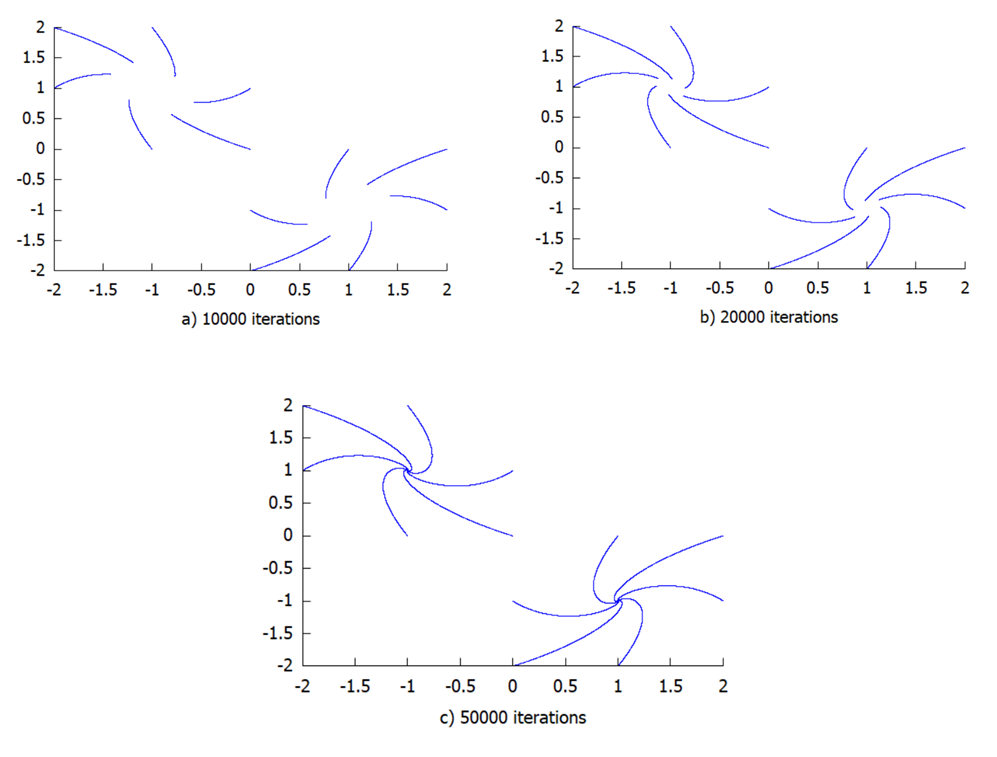}
\caption{Case $a=b=0.9998$.}
\label{Fig7}
\end{figure}

Finally, if $a=b=1$, the orbit of $(0,0)$ is $12$-periodic. On the other hand, when $a=b>1$, the orbit remains for some time near to the origin and finally goes to    infinity by the third quadrant. Also, when we increase the value of $a$, the exit from the neighborhood of $(0,0)$ is more rapid.

\begin{figure}[ht] 
\includegraphics[scale=0.7]{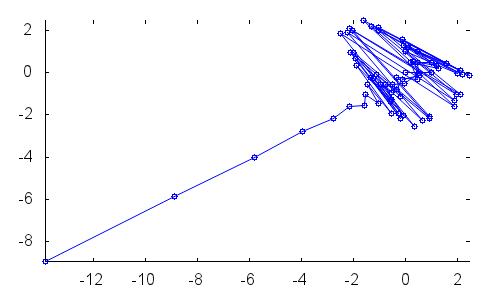}
\caption{Case $a=b=1.01$.}
\label{Fig9}
\end{figure}

It is worth mentioning that, although it seems in Figure \ref{Fig9} that the orbit accumulates in twelve arcs, in fact it tends to a $2$-periodic attractor. This is due to the fact that the eigenvalues associated to the linearization of $F^2$ at the periodic points are complex and their module is exactly $a = 0.998$, which is near $1$. This yields to a spiral convergent movement that evolves very slowly to the $2$-periodic attractor. 

\begin{figure}[ht] 
\includegraphics[scale=0.7]{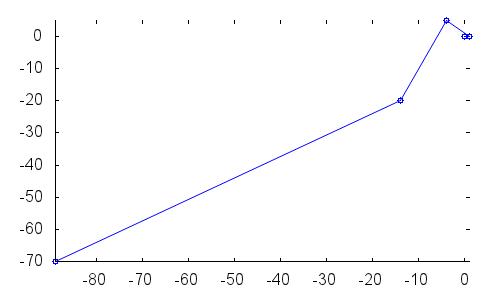}
\caption{Case $a=b=5$.}
\label{Fig10}
\end{figure}

\section{Conclusions} \label{Sec_conclusions}

The connection established between the Lozi map and difference equations of max-type allow us to extrapolate the dynamics of a single equation into a whole one-dimensional family constructed by the change of variables presented in Section \ref{Sec_transformation}. This broaden the scope of the treatment of dynamical aspects for max-type equations, usually restricted to properties on periodicity, boundedness,... In this sense, we can think in max-type equations as generators of complex dynamics (attractors, omega-limit sets,...), so their study will be interesting in the future. To this regard, it should be emphasized that the Lozi map has applications in a variety of fields, such as control theory, game theory or synchronization theory, among others (see \cite{Elhadj14}), which can be translated automatically to max equations. 

Moreover, we would like to stress that in order to consider the study of the dynamics of max-type equations we could take advantage of some techniques of differential equations or discrete dynamics, apart from the usual techniques employed in the literature that in most cases are strongly linked to arguments of real analysis. For more information related to max-type equations, see \cite{LineroN22}, a survey where the authors collect a large information on max-type difference equations and their known dynamics as well as the  techniques used in that research.

Also, we would like to highlight some open problems related to the topic treated in this paper. On the one hand, to prove analytically the dynamics of the Lozi map in the case $a=b$ when $a\notin \left[\frac{1}{2},-\frac{1}{2}\right]$, in order to see if the behavior insinuated in the numerical simulations presented in Section \ref{Sec_simulation} are true. On the other hand, to study in detail the other case deduced from the relation between the generalized Lozi map and max-type equations, namely, the case $b - a= 2$. Even, we could propose to deepen in the knowledge of the dynamics of generalized Lozi maps given by Equation~\ref{Eq_gen_Lozi} when $\delta=0$ and $\alpha+\beta+\gamma=1$.

Finally, it is worth mentioning that, nowadays, Lozi map is still a source of inspirations for research in different fields and, particularly, in the area of difference equations. 

\section*{Acknowledgements}

We would like to thank the referees for their useful comments and suggestions. This work has been supported by the grant MTM2017-84079-P funded by MCIN/AEI/10.13039/501100011033 and by ERDF ``A way of making Europe'', by the European Union.


\begin{thebibliography}{99}

\bibitem{AbuSarisA97} Abu-Saris, R., Allan, F.: Periodic and Nonperiodic Solutions of the Difference Equation $x_{n+1} = \frac{\max\{x_n^2,A \}}{x_nx_{n-1}}$, Advances in Difference Equations (Veszprem, 1995), 9-17, Gordon and Breach, Amsterdam, (1997). 

\bibitem{AbuSaris99} Abu-Saris, R.: On the periodicity of the difference equation $x_{n+1} = \alpha |x_n| + \beta x_{n-1}$, J. Difference Equ. Appl. 5, 57-69 (1999).

\bibitem{BSV09} Baptista, D.,  Severino, R.,  Vinagre, S.: The basin of attraction of Lozi mappings, Int. J. Bifurcat. Chaos 19, 1043-1049 (2009).

\bibitem{Beardon95} Beardon, A.F., Bullet, S.R., Rippon, P.J.: Periodic orbits of difference equations, Proceedings of the Royal Society of Edinburgh, 125A, 657-674 (1995).

\bibitem{BenedicksC91} Benedicks, M., Carleson, L.: The dynamics of the Hénon map, Annals of Math. 133, 73-169 (1991).

\bibitem{Botella11} Botella-Soler, V., Castelo, J.M., Oteo, J.A., Ros, J.: Bifurcations in the Lozi map, J. Phys, A: Math. Theor. 44, Art. ID 305101, 14 pages (2011).

\bibitem{Crampin92} Crampin, M.: Piecewise linear recurrence relations, The Math. Gazette 76, 355-359 (1992).

\bibitem{Devaney84} Devaney, R.L.: A piecewise linear model for the zones of instability of an area-preserving map, Phys. 10D, 387-393 (1984).

\bibitem{Elaydi05} Elaydi, S.: An introduction to difference equations. Springer, New York, (2005).

\bibitem{Elhadj14} Elhadj, Z.: Lozi mappings. Theory and applications. Taylor \& Francis, CRC Press, Boca Raton, FL (2014).

\bibitem{Feuer03} Feuer, J.: Periodic solutions of the Lyness max equation. J. Math. Anal. Appl. 288, 147-160 (2003).

\bibitem{GroveL05} Grove, E.A., Ladas, G.: Periodicities in Nonlinear Difference Equations. Chapman \& Hall, CRC Press, Boca Raton, FL (2005).

\bibitem{Henon76} Hénon, M.: A two dimensional mapping with a strange attractor, Commun. Math. Phys. 50, 69-77 (1976).

\bibitem{LineroN22} Linero-Bas, A., Nieves-Rold\'{a}n, D.: A survey on max-type equations, In: Elaydi, S., Kalabu\v{s}i\'{c}, S., Kulenovi\'{c}, M. (eds.), Advances in Discrete Dynamical Systems, Difference Equations, and Applications, Proceedings of the International Conference on Difference Equations and Applications, ICDEA 2021 (Sarajevo 2021), Springer, (In press).

\bibitem{Lorenz63} Lorenz, E.N.: Deterministic nonperiodic flow, J. Atmos. Sci. 20, 130-141 (1963).

\bibitem{Lozi} Lozi, R.: Un attracteur étrange du type attracteur de Hénon, J. Phys. 39 (1978), 9-10.

\bibitem{Misiurewicz80} Misiurewicz, M.: Strange attractor for the Lozi mapping. In: Non-linear Dynamics, Annals of the New York Academy of Sciences, 357, 348-358 (1980).

\bibitem{Ott93} Ott, E.: Chaos in dynamical systems, Cambridge University Press, New York, 1993.

\bibitem{RuelleT71} Ruelle, D., Takens, F.: On the nature of turbulence, Commun. Math. Phys. 20, 167-192 (1971).

\end{thebibliography}
\end{document}